\RequirePackage{fix-cm}
\documentclass[smallextended]{svjour3}       
\smartqed  
\usepackage{graphicx}
\usepackage{mathptmx}      
\usepackage{amsmath}
\usepackage{amssymb}
\usepackage{subcaption}
\usepackage{bm}
\usepackage{comment}
\usepackage[colorlinks=true, linkcolor=blue, citecolor=blue, urlcolor=blue, pdfborder={0 0 0}]{hyperref}
\usepackage{color}           %
\newcommand{\1}{\mbox{1}\hspace{-0.25em}\mbox{l}}
\newcommand{\rma}{\mathrm{a}}

\newcommand{\rmq}{\mathrm{q}}
\newcommand{\rms}{\mathrm{s}}
\newcommand{\st}{\mathop{\mathrm{s.{\,}t.}}}

\newtheorem{assumption}{Assumption}

\journalname{Optimization and Engineering}
\begin{document}

\title{Second-order cone programming for distributionally robust compliance optimization of trusses considering input distribution uncertainty}

\titlerunning{Second-order cone programming for distributionally robust compliance optimization...}        

\author{Takumi Fujiyama \and Yoshihiro Kanno}


\institute{
T.~Fujiyama \and Y.~Kanno
\at Department of Mathematical Informatics, Graduate School of Information Science and Technology, The University of Tokyo, Hongo  7-3-1, Bunkyo-ku, Tokyo, 1138656, Japan \\
\email{takumi007@g.ecc.u-tokyo.ac.jp}
\and Y. Kanno
\at Mathematics and Informatics Center, The University of Tokyo, Hongo 7-3-1, Bunkyo-ku, Tokyo, 1138656, Japan
}

\date{Received: date / Accepted: date}

\maketitle

\begin{abstract} 
Reliability-based design optimization (RBDO) is a methodology for designing systems and components under the consideration of probabilistic uncertainty. In practical engineering, the number of input data is often limited, which can damage the validity of the optimal results obtained by RBDO. Confidence-based design optimization (CBDO) has been proposed to account for the uncertainty of the input distribution. However, this approach faces challenges, computational cost and accuracy when dealing with highly nonlinear performance constraints. 
In this paper, we consider the compliance minimization problem of truss structures with uncertain external forces. Armed with the advanced risk measure, conditional Value-at-Risk (CVaR), we formulate a bi-objective optimization problem for the worst-case expected value and the worst-case CVaR of compliance, which allows us to account for the tail risk of performance functions not addressed in CBDO. Employing kernel density estimation for estimation of the input distribution allows us to eliminate the need for modeling the input distribution. We show that this problem reduces to a second-order cone programming when assigning either uniform kernel or triangular kernel. 
Finally, through numerical experiments, we obtain the Pareto front for the bi-objective optimization problem of the worst-case expected value and CVaR of compliance of truss structures, and confirm the changes in the Pareto solutions.


\end{abstract}

\section{Introduction} \label{1}

\subsection{Background} \label{1.1}
Design optimization of structures is a methodology to obtain a reasonable design of a structure by using a mathematical model representing the structural behavior and solving an optimization problem based on the model. 
The design domain and input data such as the external forces 
are usually given, and the optimal design is often obtained by 
maximizing a performance function of a structure under the constraint of 
an upper bound on the volume of the structure. 

Structures in the real-world are subjected to various uncertainties, 
such as external forces that cannot be anticipated during the design 
phase and variations in material properties that occur during the 
manufacturing process. Structures designed using deterministic models 
can possibly be vulnerable to these uncertainties in input data, which may result in compromised safety and performance in real-world environments. Therefore, it is pivotal to consider the uncertainties in input data during the design phase and perform optimization to ensure the safety and high-performance of designed structures.

{\em Reliability-based design optimization\/} (RBDO) \cite{Acar2021,DerKilureghian22} is one of the probabilistic methods that considers such uncertainty in design optimization of structures. In RBDO, uncertain parameters are treated as random variables that follow a certain probability distribution. Since the values of the performance function of structures can be regarded as random variables, uncertainties are considered by imposing constraints on the reliability, that is, the probability that the values of the performance function satisfy the constraint conditions. 
Recent studies have proposed scenario optimization approaches to RBDO that address the challenges of probabilistic constraint satisfaction in the presence of epistemic uncertainty. For example, Rocchetta et al. \cite{RC21} introduced a framework based on sampling and scenario-based formulation to ensure reliability with finite samples. An extension of this work was later presented in \cite{RCK20}, where the authors proposed a soft-constrained modulation of failure probability bounds to handle risk and reliability in a unified manner.
Recently, there has been significant progress in the study of RBDO that 
also considers the uncertainty in the probability distributions of 
uncertain parameters, called {\em confidence-based design optimization\/} 
(CBDO) \cite{MCCGLG18,JCL19,WHYWG20}. When sufficient input data is 
not available, it is difficult to accurately estimate the input 
probability distribution, making it challenging to properly evaluate the 
structural reliability. As a result, designs obtained through RBDO can become unsafe. In CBDO, the uncertainty in the probability distribution is considered by imposing constraints on the confidence, which is the probability that the reliability constraints are satisfied.
Moon \textit{et al.}\ \cite{MCCGLG18} proposed a method based on a 
hierarchical Bayesian model and the {\em Monte Carlo simulation\/} (MCS). 
This method features a double-loop MCS, where the inner loop is for the 
reliability assessment and the outer loop is for the confidence 
assessment. This double-loop structure results in a very high computational cost. Additionally, it assumes a specific type of input distribution, which can lead to an unreliable design if the assumed distribution type differs from the true distribution. Jung \textit{et al.}\ \cite{JCL19} proposed a method that replaces the MCS-based 
reliability assessment with the {\em first-order reliability method\/} 
(FORM). FORM approximates the reliability by linearizing the performance 
function, which can reduce accuracy for a highly nonlinear performance 
function. This method still retains a double-loop MCS structure with the 
outer loop for the confidence assessment, and thereby maintains high 
computational cost. Wang \textit{et al.}\ \cite{WHYWG20} proposed a 
single-loop method using the {\em second-order reliability method\/} 
(SORM) for the reliability assessment, which improves accuracy through 
the second-order approximation and reduces computational cost. Jung 
\textit{et al.}\ \cite{JCDL20} proposed a method to estimate the optimal 
number of samples by considering the costs of optimizing the design 
variables and adding new samples. Hao \textit{et al.}\ \cite{HYYZWW22} 
achieved significant computational cost reduction by decoupling the 
reliability optimization and the confidence analysis. Furthermore, Jung 
\textit{et al.}\ \cite{JKCHL24} proposed a sampling-based RBDO method 
that evaluates the confidence by estimating the input distributions using 
the multivariate kernel density estimation on bootstrap samples of the 
input data and calculating the reliability under each of the estimated  distributions.

In contrast, in the field of mathematical optimization, {\em robust 
optimization\/} (RO) \cite{BN98} is used as an optimization 
framework that handles uncertainties. In RO, an uncertainty set is 
constructed as a set of the possible values of uncertain parameters, and 
the uncertainty is considered by solving a two-stage optimization 
problem that seeks the optimal decision variables for the worst-case 
uncertain parameter values within this uncertainty set.
As an application of RO in the field of structural reliability, Ben-Tal 
and Nemirovski \cite{BN02} optimized the design of antennas. Moreover, 
Kanno and Takewaki \cite{KT06} proposed a method for the optimal design 
of trusses with uncertain external forces under a volume constraint, describing it as a nonlinear semidefinite programming problem and approximating it with sequential semidefinite 
programming. Furthermore, Kanno \cite{Kanno19,Kanno20} 
proposed a robust optimization problem that conservatively approximates 
CBDO by constructing an uncertainty set for external forces using the order statistics, and minimized the volume of structures under compliance constraints.

{\em Distributionally robust optimization\/} (DRO) \cite{DY10,GS10} is 
one of the models that handle uncertainties and have been actively 
studied in recent years. While RO constructs an uncertainty set for the 
possible values of uncertain parameters, DRO treats uncertain parameters 
as random variables and constructs an uncertainty set, called the 
{\em distributionally uncertainty set\/}, for the 
probability distributions that they follow. By incorporating probabilistic aspects, DRO overcomes the excessive conservativeness of optimal solutions that is a challenge in RO. Several methods have been proposed for constructing the distributionally uncertainty set. Delage and Ye \cite{DY10} proposed a moment-based method. They construct an ellipsoid using the sample mean and sample variance-covariance matrix calculated from samples of uncertain parameters, and define the uncertainty set as the set of all distributions that share the first and second moments within this ellipsoid. Mohajerin Esfahani and Kuhn \cite{MK18} proposed a method based on the Wasserstein metric. They define the uncertainty set as the set of all distributions whose Wasserstein distance from the empirical distribution is within a specified threshold. Moreover, Bertsimas \textit{et al.}\ \cite{BGK18} proposed a method that uses goodness-of-fit tests to define the uncertainty set as the set of all distributions that do not reject the null hypothesis of being the true distribution. As an application of DRO in the field of structural reliability, Kanno \cite{Kanno22} constructed an uncertainty set of distributions based on moments and formulated the problem of minimizing the volume of trusses under compliance constraints as a nonlinear semidefinite programming problem. Furthermore, Chen \textit{et al.}\ \cite{CRKGY24} applied DRO to the aerodynamic shape optimization of airfoils for transonic speeds.

In financial engineering, a recently used risk measure is 
{\em conditional value-at-risk\/} (CVaR) \cite{RU00}.  
A reliability constraint in RBDO is often interpreted as controlling a particular quantile—also known as the {\em value-at-risk} (VaR)—of the performance function. CVaR is known as a convex and conservative approximation of VaR \cite{RU02}, and is a risk measure that considers the average outcome in the worst-case scenarios beyond the VaR. Rather than just ensuring that failures are rare, CVaR focuses on how severe those failures can be, offering a more conservative and risk-averse criterion. This distinction is particularly important when designing structures that must maintain their functionality and safaty under extreme conditions. The limitations of VaR and the advantages of CVaR in addressing tail risk have been well documented in the literature, including the work by  Embrechts {\em et al}. \cite{ESW21}. 
Furthermore, CVaR possesses a theoretically desirable property of coherence \cite{RU02} and is a convex risk measure, which makes it tractable in optimization \cite{NS06}.
Recently, research has been conducted on the use of CVaR in the field of 
structural reliability \cite{RR10,BR22,BOR23,CNK20,CKNRW22}. By 
using CVaR as a constraint function, it is possible to evaluate the tail 
of the distribution of structural performance, which cannot be assessed 
by reliability constraints used in RBDO and CBDO. 
Evaluation of the tail allows for risk assessment that considers the probability of significant performance or safety degradation. Rockafellar and Royset \cite{RR10} and Byun and Royset \cite{BR22} have shown that using CVaR constraints instead of reliability constraints can lead to more conservative designs. Byun \textit{et al.}\ \cite{BOR23} considered a penalty function using CVaR for a first-order approximated performance function and formulated the unconstrained optimization problem obtained by the penalty method as an approximate problem of the optimal design problem under CVaR constraints. Chaudhuri \textit{et al.}\ \cite{CNK20,CKNRW22} compared the differences between reliability constraints and CVaR constraints in terms of formulations and numerical experiments.
  
\subsection{Objective} \label{1.2}
In structural optimization, we model the uncertainty of external forces as a random vector. We consider a bi-objective optimization of the expected value and CVaR, with compliance of trusses as the measure of structural performance. 
By using the kernel density estimation, we can formulate the 
optimization problem from samples of external forces without the need to 
model the input distribution. Considering the uncertainty of the input 
distribution, we formulate the worst-case CVaR constrained worst-case expected value minimization problem. We reduce the problem to a convex optimization problem. 
We show that this convex optimization problem can be recast as a {\em 
second-order cone programming\/} problem when the kernel function is either a uniform kernel or a triangular kernel.

\subsection{Contributions} \label{1.3}
The relationship to prior studies and our original contributions are summarized as follows:

Our framework builds upon recent developments in the field of portfolio optimization in finance, where uncertainty sets are constructed using the kernel density estimation combined with the $\phi$-divergence. These methods estimate the underlying distribution in a non-parametric manner and define a robustness region over sample weights. Motivated by this idea, we extend the kernel density estimation and the $\phi$-divergence formulation to the context of structural optimization under uncertain external load. 

The proposed approach is partially data-driven. The input distribution is constructed in a fully non-parametric and data-driven manner without assuming any predefined distributional form. However, the uncertainty set over the sample weights is not data-driven, as its size is determined by a level of uncertainty. This parameter is not estimated from the data, but rather chosen in advance to reflect the desired level of conservativeness. Therefore, while the input distribution reflects the empirical data, the extent of distributional robustness is model-based.

Building on this foundation, our main contributions are as follows: 
\begin{itemize}
  \item We propose a bi-objective optimization problem that simultaneously minimizes the worst-case expected value and the worst-case CVaR of compliance. This allows for integrated control of both average and tail structural risks.
 
  \item For uniform and triangular kernels, we derive explicit second-order cone programming formulations, enabling efficient solution via the primal-dual interior-point methods with polynomial-time complexity.
\end{itemize}
These features distinguish our method from conventional robust or reliability-based design approaches, and enable risk-aware, tractable optimization under epistemic uncertainty.

\subsection{Oganization} \label{1.4}
The paper is organized as follows. In Section \ref{2}, we review the work of Liu \textit{et al.}\ \cite{LYY22} who proposed a distributionally robust optimization problem based on the kernel density estimation and reduced it to a single-stage optimization.
In Section \ref{3}, we summarize the compliance minimization problem of 
trusses addressed in this paper. Then, we formulate a bi-objective 
optimization problem of the worst-case expected value and the worst-case 
CVaR of the compliance by using the distributionally robust optimization prepared in Section \ref{2}. We reformulate it as the worst-case CVaR constrained the worst-case expected value minimization via the $\varepsilon$-constraint method, and show that it is a convex optimization problem. 
In Sections \ref{4} and \ref{5}, we show that the problems 
formulated in Section \ref{3} can be reduced to a second-order cone 
programming problem when adopting either the uniform kernel or the triangular kernel as a kernel function, respectively.
Section \ref{6} reports the results of numerical experiments on the optimization problems formulated in Sections \ref{4} and \ref{5}.
Finally, Section \ref{7} concludes the paper.

\subsection{Notation} \label{1.5}
In our notation, ${}^{\top}$ denotes the transpose of a vector or a matrix. 
Let the set of $m$-dimensional real vectors with strictly positive components and the set of $m$-dimensional real vectors with non-negative components are denoted as $\mathbb{R}_{++}^m$ and $\mathbb{R}_{+}^m$, respectively. 
We use $\mathcal{S}_+^m$ to denote the set of $m \times m$ symmetric 
positive semidefinite matrices. The zero vectors are denoted as $\bm{0}$, 
and let $\bm{1}={(1, \dots, 1)}^\top$. The Euclidean norm of a vector is 
denoted by $\| \cdot \|$. We use $\mathbb{E}_P [\,\cdot\,]$ to denote the expected value under the distribution $P$. Define 
${[\,\cdot\,]}^+ : \mathbb{R} \to \mathbb{R}$ by 
${[\,\cdot\,]}^+ = \max \{ 0, \,\cdot\, \}$. We use $N(\bm{\mu}, \varSigma)$ to denote the multivariate normal distribution with mean vector $\bm{\mu}$ and variance-covariance matrix $\varSigma$. For a function $\phi: \mathbb{R}^m \to \mathbb{R} \cup \{ +\infty \}$, its conjugate function $\phi^*: \mathbb{R}^m \to \mathbb{R} \cup \{ +\infty \}$ is defined by
\begin{align*}
\phi^*(\bm{s}) = \sup \{\bm{s}^\top \bm{t}-\phi(\bm{t}) \}.
\end{align*}

\section{Preliminaries} \label{2}
This section summarizes Liu \textit{et al.} \cite{LYY22}, who proposed a distributionally robust optimization problem based on kernel 
density estimation and reduced it to a single-stage optimization problem. 
This is intended to prepare for the formulations in Section \ref{3}.

\subsection{Kernel density estimation based distributionally robust optimization problem} \label{2.1}
Let $\bm{x} \in \mathcal{X} \subseteq \mathbb{R}^m$ denote the vector of 
design variables, where $m$ is the number of independent design 
variables and $\mathcal{X}$ is a non-empty convex set. 
We use $\bm{\xi} \in \mathbb{R}^d$ to denote a random variable vector, where $d$ is the dimension of $\bm{\xi}$. More precisely, 
$\bm{\xi}$ is defined on a probability space $(\varOmega, \mathcal{F}, \mathbb{P}), $ where $\varOmega$ is a compact sample space, $\mathcal{F}$ is a $\sigma$-algebla of $\varOmega$, and $\mathbb{P}$ is a probability measure on measurable space $(\varOmega, \mathcal{F})$. Furthermore, let $f: \mathbb{R}^{m}\times\mathbb{R}^{d}\to \mathbb{R}$ be a performance function of a structure. Then, performance function values of the structure $f(\bm{x}; \bm{\xi})$ can be regarded as a random variable.

For a given set of samples $\{ {\hat{\bm{\xi}}_1}, \dots, {\hat{\bm{\xi}}_n} \}$ of the random variable $\bm{\xi}$ drawn from distribution $P$, we can obtain 
a set of the corresponding structural performance values of the structure $\{ f(\bm{x}; \hat{\bm{\xi}}_1), \dots, f(\bm{x}; \hat{\bm{\xi}}_n) \}$. Here, $n$ denotes the number of samples.
Following the methodology outlined in \cite{LYY22}, we estimate 
distribution $P$ by the weighted kernel density estimator 
\begin{align}
  \hat{p}_{\bm{w}} (y) = \frac{1}{h} \sum_{i=1}^{n} w_i k \left( \frac{y - f(\bm{x}; \hat{\bm{\xi}}_i)}{h}\right), \label{WKDE}
\end{align}
and construct a distributionally uncertainty set. 
In this expression, the weight vector is $\bm{w} = {(w_1, \dots, 
w_n)}^\top \in \mathcal{W}$, where $\mathcal{W}\subset\mathbb{R}^n$ is defined by
\begin{align*}
  \mathcal{W}= \{ \bm{w} \in \mathbb{R}^n \mid \bm{1}^\top \bm{w} =1,\ \bm{w} \geq \bm{0} \}.
\end{align*}
Moreover, $k: \mathbb{R} \to \mathbb{R}_+$ is a kernel function, and $h\in\mathbb{R}_{++}$ is a constant representing the bandwidth of the kernel function. In this paper, we assume that kernel function $k$ satisfies Assumption \ref{assumption1} and Assumption \ref{assumption2}. 
\begin{assumption} \label{assumption1}
$k$ is a bounded non-negative function satisfying $\int_{-\infty}^\infty k(y)\textup{d}y=1$. 
\end{assumption}
\begin{assumption} \label{assumption2}
$k(y) = k(-y)$ is satisfied for any $y \in \mathbb{R},$ and $\int_{-\infty}^{\infty} y^2 k(y) \textup{d}y < \infty$ holds.
\end{assumption}

The $\phi$-divergence between two weight vectors $\bm{w}, \bm{w}^0 \in \mathbb{R}^n$ is defined by
\begin{align}
  I_\phi(\bm{w}, \bm{w}^0) = \sum_{i=1}^n w_i^0 \phi \left( \frac{w_i}{w_i^0} \right). \label{phi-divergence}
\end{align}
Here, $\phi: \mathbb{R} \to \mathbb{R} \cup \{+\infty\}$ satisfies 
$\phi(1) = 0$, $\phi(t) = + \infty$ for any $t<0$, 
and is convex on $\mathbb{R}_{+}$. Furthermore, for $a>0$, we define $0\phi(a/0) := a \displaystyle \lim_{t\to \infty}\phi(t)/t$ and $0\phi(0/0) :=0$. 
With parameter $\tau \in \mathbb{R}_{+}$ representing the level of uncertainty, 
we define the uncertainty set for the weights $\mathcal{W}_\phi^\tau 
\subset \mathbb{R}^n$ by 
\begin{align}
  \mathcal{W}_\phi^\tau = \left\{ \bm{w} \mid \bm{w} \in \mathcal{W},\ I_\phi(\bm{w}, \bm{w}^0) \leq \tau \right\}, \label{W_phi^tau}
\end{align}
which is called the {\em weight uncertainty set\/}. 
Here, the weight vector $\bm{w}^0$, which represents the center of 
$\mathcal{W}_\phi^\tau$, is chosen such that $\bm{w}^0 \in \mathcal{W}$. 
In the absence of information about the distribution $P$,  a uniform 
weight $\bm{w}^0 = n^{-1}\bm{1}$ is generally adopted. The uncertainty 
set of the distribution that $f(\bm{x}; \bm{\xi})$ belongs to is 
constructed using the weighted kernel density estimator \eqref{WKDE} and 
$\mathcal{W}_\phi^\tau$ as
\begin{align}
  \mathcal{P}_{\mathcal{W}_{\phi}^\tau} := \left\{ {\hat{p}}_{\bm{w}}(y) = \frac{1}{h} \sum_{i=1}^n w_i k\left( \frac{y - f(\bm{x}; \hat{\bm{\xi}}_i)}{h} \right) \middle| \bm{w} \in \mathcal{W}_\phi^\tau \right\}, \label{KDE-based DUS}
\end{align}
which is called the {\em distributionally uncertainty set\/}. 
Using \eqref{KDE-based DUS}, we formulate the distributionally robust expected value minimization problem as 
\begin{alignat}{3}
  &\underset{\bm{x} \in \mathcal{X}}{\textrm{Min.}} &\hspace{1mm} &\underset{P \in \mathcal{P}_{\mathcal{W}_\phi^\tau}}{\max} \left\{ \mathbb{E}_P \left[ f(\bm{x}; \bm{\xi}) \right] \right\}. \label{Main}
\end{alignat}
Since each kernel is symmetric around its corresponding sample point 
due to Assumption \ref{assumption2}, we can express problem \eqref{Main} 
using the weight uncertainty set $\mathcal{W}_\phi^\tau$ as
\begin{alignat}{3}
  &\underset{\bm{x} \in \mathcal{X}}{\textrm{Min.}} &\hspace{1mm} &\underset{\bm{w} \in \mathcal{W}_\phi^\tau}{\max} \left\{ \sum_{i=1}^n w_i f(\bm{x}; \hat{\bm{\xi}}_i) \right\}. \label{Main_weight}
\end{alignat}

\subsection{Reduction to a single-stage optimization problem} \label{2.2}
For $\tau=0$, we have $\mathcal{W}_{\phi}^{\tau} = \{ \bm{w}^0 \}$, 
thereby problem \eqref{Main_weight} can be expressed as
\begin{align*}
  \underset{\bm{x} \in \mathcal{X}}{\textrm{Min.}}\quad \displaystyle \sum_{i=1}^n w_i^0 f(\bm{x}; \bm{\xi}).
\end{align*} 

For $\tau > 0$, we assume that $\tau$ is sufficiently small such that $\bm{w} \in \mathcal{W}_{\phi}^{\tau}$ is satisfied. 
It is shown in \cite{LYY22} that 
problem \eqref{Main_weight} can be reduced to a single-stage optimization problem. 
The Lagrangian $L(\bm{x}; \,\cdot\,, \,\cdot\,, \,\cdot\,): \mathbb{R}^n \times \mathbb{R} \times \mathbb{R} \to \mathbb{R}$ for the inner maximization problem of \eqref{Main_weight} is given by
\begin{align*}
  L(\bm{x}; \bm{w}, \lambda, \eta) 
  &= \sum_{i=1}^n w_i f(\bm{x}; \hat{\bm{\xi}}_i) + \lambda \left( \tau - \sum_{i=1}^n  w_i^0 \phi \left( \frac{w_i}{w_i^0} \right) \right) +\eta \left( 1 - \bm{1}^\top \bm{w} \right) \\
  &= \tau \lambda + \eta + \sum_{i=1}^n \left[ w_i \left( f(\bm{x}; \hat{\bm{\xi}}_i) - \eta \right) - \lambda w_i^0 \phi \left( \frac{w_i}{w_i^0} \right) \right],
\end{align*}
where $\lambda \in \mathbb{R}_+$ and $\eta \in \mathbb{R}$ are Lagrange 
multipliers. Therefore, the objective function of the dual problem, 
denoted by $d (\bm{x};\,\cdot\,,\,\cdot\,)$, is obtained as 
\begin{align*}
  d (\bm{x}; \lambda, \eta) 
  &:= \sup_{\bm{w} \geq \bm{0}} \left\{ L(\bm{x}; \bm{w}, \lambda, \eta) \right\} \\
  &= \tau \lambda + \eta + \sup_{\bm{w} \geq \bm{0}} \left\{\sum_{i=1}^n \left[ w_i \left( f(\bm{x}; \hat{\bm{\xi}}_i) - \eta \right) - \lambda w_i^0 \phi \left( \frac{w_i}{w_i^0} \right) \right] \right\} \\
  &= \tau \lambda + \eta + \sum_{i=1}^n \sup_{w_i \geq 0} \left\{ w_i \left( f(\bm{x}; \hat{\bm{\xi}}_i) - \eta \right) - \lambda w_i^0 \phi \left( \frac{w_i}{w_i^0} \right) \right\} \\
  &= \tau \lambda + \eta + \lambda \sum_{i=1}^n w_i^0 \sup_{t_i \in \mathbb{R}} \left\{ \frac{t_i}{\lambda} \left( f(\bm{x}; \hat{\bm{\xi}}_i) - \eta \right) - \phi (t_i) \right\} \\
  &= \tau \lambda + \eta + \lambda \sum_{i=1}^n w_i^0 \phi^* \left( \frac{f(\bm{x}; \hat{\bm{\xi}}_i) - \eta}{\lambda} \right).
\end{align*}
Since $\mathcal{W}_\phi^\tau$ is a convex set, the inner maximization problem of \eqref{Main_weight} is a convex optimization problem. 
Moreover, since $\bm{w}^0 \in \mathcal{W}$ and $I_\phi \left( \bm{w}^0, 
\bm{w}^0 \right)=0 < \tau$, $\bm{w}^0$ is a relative interior point of 
$\mathcal{W}_\phi^\tau$. Therefore, $\mathcal{W}_\phi^\tau$ satisfies 
the Slater constraint qualification for convex optimization, which 
ensures the strong duality \cite[Proposition 8.7]{CE14}. 
Accordingly, we have 
\begin{align*}
  \underset{\bm{w} \in \mathcal{W}_\phi^\tau}{\max} \left\{ \sum_{i=1}^n w_i f(\bm{x}; \hat{\bm{\xi}}_i) \right\}
  = \underset{\lambda \geq 0, \eta \in \mathbb{R}}{\min} \{ d(\bm{x}; \lambda, \eta) \}.
\end{align*}
Consequently, problem \eqref{Main_weight} reduces to the single-stage optimization problem
\[
\begin{array}{ll}
\textrm{Min.} & d(\bm{x}; \lambda, \eta)  \\
   \st & \bm{x} \in \mathcal{X},\quad \lambda \geq 0, 
\end{array}
\]
which can be expressed explicitly as
\[
\begin{array}{ll}
    \textrm{Min.} & \tau \lambda + \eta + \lambda \displaystyle \sum_{i=1}^n w_i^0 \phi^* \left( \frac{f(\bm{x}; \hat{\bm{\xi}}_i) - \eta}{\lambda} \right)  \\
   \st & \bm{x} \in \mathcal{X},\quad \lambda \geq 0, \label{one-step}
\end{array}
\]
where the optimization variables are $\bm{x} \in \mathbb{R}^m$, $\lambda 
\in \mathbb{R}$, and $\eta \in \mathbb{R}$.

Since we assumed $\phi(t) = +\infty$ for any $t<0$, we have 
$\phi^*(s) = \underset{t \geq 0}{\sup} \{st - \phi(t)\}$, 
which is the pointwise supremum of affine functions with non-negative 
slopes. This shows that $\phi^*$ is a non-decreasing function. Therefore, 
for each $i=1,\dots,n$,
\begin{align*}
  \iota_i \geq f(\bm{x}; \hat{\bm{\xi}}_i) - \eta
\end{align*}
implies
\begin{align*}
  \phi^* \left( \frac{\iota_{i}}{\lambda} \right) \geq \phi^* \left( \frac{f(\bm{x}; \hat{\bm{\xi}}_i) - \eta}{\lambda} \right).
\end{align*}
Thus, problem \eqref{one-step} can be rewritten as
\begin{alignat*}{3}
  &\textrm{Min.} &\hspace{1mm} & \tau \lambda + \eta + \lambda \sum_{i=1}^n w_i^0 \phi^* \left( \frac{\iota_{i}}{\lambda} \right) \\
  &\st &\quad & \iota_{i} \geq f(\bm{x}; \hat{\bm{\xi}}_i) - \eta,\quad i=1,\dots,n, \\
  & &\quad &\bm{x} \in \mathcal{X},\quad \lambda \geq 0, 
\end{alignat*}
where the optimization variables are $\bm{x} \in \mathbb{R}^m$, $\lambda \in \mathbb{R}$, $\eta \in \mathbb{R}$, and $\bm{\iota} \in \mathbb{R}^n$.

\section{Formulation of the bi-objective optimization problem for truss structures} \label{3}
In this section, we summarize the compliance minimization problem, which is one of the topology optimization problems for truss structures. Subsequently, using the distributionally robust optimization based on kernel density estimation prepared in Section \ref{2}, we formulate a bi-objective optimization problem for the worst-case expected value and the worst-case Conditional Value-at-Risk of compliance of truss structures.

\subsection{Compliance minimization problem of trusses} \label{3.1}
Let $m$ denote the number of members of the truss. 
The design variable is the vector of member cross-sectional areas, 
denoted by $\bm{x} \in \mathcal{X} \subseteq \mathbb{R}^m$, and 
$\mathcal{X}$ is the admissible set of $\bm{x}$ defined by 
\begin{align}
  \mathcal{X} = \{ \bm{x} \in \mathbb{R}^m \mid \bm{x} \geq \bm{0},\ \bm{l}^\top \bm{x} \leq \overline{V} \}, \label{X}
\end{align}
where $\bm{l} \in \mathbb{R}_{++}^m$ is the vector of member lengths, 
and $\overline{V} > 0$ is the upper limit on the total structural volume 
of the truss. 

Let $d$ denote the number of degrees of freedom of the nodal 
displacements. 
We use $K(\bm{x}) \in \mathcal{S}_+^d$ to denote the stiffness matrix, 
which can be expressed as 
\begin{align}
  K(\bm{x})
  &= \sum_{j=1}^{m} \frac{Ex_j}{l_j} \bm{\beta}_j \bm{\beta}_j^\top, \label{decomp}
\end{align}
where $\bm{\beta}_1, \dots, \bm{\beta}_m \in \mathbb{R}^d$ are constant 
vectors, and $E$ is the Young modulus. 
The compliance is a measure of the static flexibility of a structure, 
and is defined by 
\begin{align}
\pi^\textrm{c} (\bm{x}; \bm{\xi}) = \sup_{\bm{u} \in \mathbb{R}^d} \left\{ 2 \bm{\xi}^\top \bm{u} - {\bm{u}}^\top K(\bm{x}) \bm{u} \right\},\label{compliance}
\end{align}
where $\bm{u} \in \mathbb{R}^d$ is the nodal displacement vector, and 
$\bm{\xi} \in \mathbb{R}^d$ is the random variable vector representing 
the static external load. 

As formally stated in Proposition \ref{Compliance_can_be_SOCP}, 
it is known that the compliance minimization problem of trusses can be 
reduced to a second-order cone programming problem. 

\begin{proposition} \textup{[\cite{BN01}; see also \cite{MBK09}]} \label{Compliance_can_be_SOCP}
  It is equivalent that $\bm{x}^*$ is the optimal value of the compliance minimization problem for trusses: 
\[
\begin{array}{lll}
    &\textup{Min.} &\pi^\textup{c} (\bm{x}; \bm{\xi}) \\
    &\st &\bm{x} \in \mathcal{X}
\end{array}
\]
and there exists variables $\bm{b} \in \mathbb{R}^m, \bm{q} \in \mathbb{R}^m$ such that $( \bm{b}^*, \bm{q}^*, \bm{x}^* )$ is the optimal value of the problem
\begin{equation}
    \begin{array}{ll}
    \textup{Min.} & \displaystyle \sum_{j=1}^m 2b_j \\
    \st &  b_j + x_j \geq 
    \left\| 
    \begin{bmatrix}
        b_j - x_j \\
        \sqrt{2l_j/E} q_j
    \end{bmatrix}
    \right\|, \quad j=1,\dots,m, \\
    & \displaystyle \sum_{j=1}^m q_j \bm{\beta}_j = \bm{\xi}, \\
    & \bm{x} \in \mathcal{X}, \label{compliance minimization is SOCP}
    \end{array}
\end{equation}
where the optimization variables are $\bm{x} \in \mathcal{X}, \bm{b} \in \mathbb{R}^m, \bm{q} \in \mathbb{R}^m$.
\end{proposition}

\subsection{Distributionally Robust bi-objective optimization problem with expected value and CVaR} \label{3.2}
For a given design variable vector $\bm{x} \in \mathbb{R}^m$, 
confidence level $\gamma \in [0, 1)$, and the distribution $P$, the {\em value-at-risk\/} (VaR), introduced by \cite{VaR}, of $\pi^\textrm{c}(\bm{x}; \bm{\xi})$ is defined by
\begin{align}
\text{{VaR}}_P^\gamma (\pi^\textrm{c}(\bm{x}; \bm{\xi})) = \min \left\{\alpha \in \mathbb{R} \mid \mathbb{P} \{ \pi^\textrm{c}(\bm{x}; \bm{\xi}) \leq \alpha \} \geq \gamma \right\}. \label{VaR}
\end{align}
Moreover, the {\em conditional value-at-risk\/} (CVaR), introduced by 
Rockafellar and Uryasev \cite{RU00}, is defined by the expected value of 
the tail loss exceeding the VaR, i.e., 
\begin{align}
\text{{CVaR}}_P^\gamma (\pi^\textrm{c}(\bm{x}; \bm{\xi})) = \text{{VaR}}_P^\gamma (\pi^\textrm{c}(\bm{x}; \bm{\xi})) + \frac{1}{1-\gamma}  \mathbb{E}_P \left[ {[\pi^\textrm{c}(\bm{x}; \bm{\xi}) - {\text{VaR}}_P^\gamma (\pi^\textrm{c}(\bm{x}; \bm{\xi}))]}^+ \right]. \label{CVaR}
\end{align}
It is known that 
CVaR is the tightest convex upper bound of VaR \cite{NS06}, and satisfies the property of coherence, which is a desirable characteristic for a risk measure \cite{RU02}. 

Consider the following bi-objective distributionally robust optimization problem of the worst-case expected value and the worst-case CVaR of compliance:
\begin{align}
    \underset{\bm{x} \in \mathcal{X}}{\text{Min.}} \quad 
    \bigg( \max \left\{ \mathbb{E}_P \left[ \pi^{\text{c}}(\bm{x}; \bm{\xi}) \right] \, \middle| \, P \in \mathcal{P}_{\mathcal{W}_\phi^\tau} \right\}, 
    \max \left\{ \text{CVaR}_P^\gamma (\pi^{\text{c}}(\bm{x}; \bm{\xi})) \, \middle| \, P \in \mathcal{P}_{\mathcal{W}_\phi^\tau} \right\} \bigg).
\end{align}
To find the Pareto solutions of this problem, we formulate the optimization problem of the worst-case expected value minimization under a worst-case CVaR constraint using the $\varepsilon$-constraint method 
as follows: 
\begin{equation}
  \begin{array}{lll}
    & \underset{\bm{x} \in \mathcal{X}}{\textrm{Min.}} & \underset{P \in \mathcal{P}_{\mathcal{W}_\phi^\tau}}{\max} \left\{ \mathbb{E}_P \left[ \pi^{\text{c}}(\bm{x}; \bm{\xi}) \right] \right\} \\
    & \st & \underset{P \in \mathcal{P}_{\mathcal{W}_\phi^\tau}}{\max} \left\{ \text{{CVaR}}_P^\gamma (\pi^\textrm{c}(\bm{x}; \bm{\xi})) \right\} \leq \nu.
  \end{array} \label{ECVaR}
\end{equation}
Here, $\nu \in \mathbb{R}$ is a parameter for the $\varepsilon$-constraint method. 

Define $F_P^\gamma: \mathbb{R}^m \times \mathbb{R} \to \mathbb{R}$ by 
\begin{align}
  F_P^\gamma (\bm{x}, \alpha) = \alpha + \frac{1}{1-\gamma} \mathbb{E}_P 
  \left[ {\left[ \pi^\textrm{c}(\bm{x}; \bm{\xi}) - \alpha \right]}^+ 
  \right], \label{F_P^gamma}
\end{align}
Rockafellar and Uryasev \cite[Theorem~2]{RU00} show that 
CVaR in \eqref{CVaR} can be expressed as
\begin{align*}
\textup{CVaR}_P^\gamma (\pi^\textrm{c}(\bm{x}; \bm{\xi})) = \underset{\alpha \in \mathbb{R}}{\min} \left\{ F_P^\gamma (\bm{x}, \alpha) \right\}.
\end{align*}
Therefore, problem \eqref{ECVaR} can be rewritten equivalently as
\begin{equation}
  \begin{array}{lll}
    & \underset{\bm{x} \in \mathcal{X}, \alpha \in \mathbb{R}}{\textrm{Min.}} & \underset{P \in \mathcal{P}_{\mathcal{W}_\phi^\tau}}{\max} \left\{ \mathbb{E}_P \left[ \pi^\textrm{c}(\bm{x}; \bm{\xi}) \right] \right\} \\
    & \st & \underset{P \in \mathcal{P}_{\mathcal{W}_\phi^\tau}}{\max} \left\{ F_P^\gamma (\bm{x}, \alpha) \right\} \leq \nu.
  \end{array} \label{Original}
\end{equation}

Given a set of samples $\{ \hat{\bm{\xi}}_1, \dots, \hat{\bm{\xi}}_n \}$ 
of random variable $\bm{\xi}$ and a set of the corresponding samples of 
the compliance $\{ \pi^\textrm{c}(\bm{x}; \hat{\bm{\xi}}_1), \dots, \pi^\textrm{c}(\bm{x}; \hat{\bm{\xi}}_n) \}$, $F_P^\gamma (\bm{x}, \alpha)$ in \eqref{F_P^gamma} can be approximated using the weighted kernel density estimator \eqref{WKDE} as \cite{LYY22}
\begin{align}
  F_P^\gamma (\bm{x}, \alpha) 
  &= \alpha + \frac{1}{1-\gamma} \sum_{i=1}^n w_i h \psi_k \left( \frac{\pi^\textrm{c}(\bm{x}; \hat{\bm{\xi}}_i) - \alpha}{h} \right). \label{F_P^gamma_approx}
\end{align}
Here, $\psi_k : \mathbb{R} \to \mathbb{R}$ is defined by 
\begin{align}
  \psi_k(c)&=cG_k(c)-\tilde{G}_k(c) \label{psi_K}
\end{align}
with 
\begin{align}
  G_k(c)&:=\int_{-\infty}^c k(y)\text{d}y, \label{G_k} \\
  \tilde{G}_k(c)&:=\int_{-\infty}^c yk(y)\text{d}y, \label{tildeG_k}
\end{align}
where $k$ in \eqref{G_k} and \eqref{tildeG_k} is the kernel function, 
and $h\in\mathbb{R}_{++}$ and $\bm{w} \in \mathcal{W}$ in 
\eqref{F_P^gamma_approx} are the bandwidth and weight vector of the kernel 
function, respectively, as introduced in \eqref{WKDE}. 
Define $\Upsilon_k: \mathbb{R} \to \mathbb{R}$ by 
\begin{align}
  \Upsilon_k(c)
    & = h\psi_k(h^{-1}c) \label{Upsilon} \\
    &= c \int_{-\infty}^\frac{c}{h} k(y)\text{d}y - h\int_{-\infty}^\frac{c}{h} yk(y)\text{d}y, \label{Upsilon_decomp}
\end{align}
where the last equality follows from \eqref{psi_K}, \eqref{G_k}, and 
\eqref{tildeG_k}. 
The following proposition is known to hold for $\Upsilon_k$.
\begin{proposition} \textup{\cite[Propositoin~1]{LYY22}} \label{Upsilon_Prop} \\
Under Assumptions \ref{assumption1} and \ref{assumption2}, $\Upsilon_k$ 
  defined by \eqref{Upsilon} satisfies the following properties: \\
  $(a)$ For any $c \in \mathbb{R}$, $\frac{\textup{d}}{\textup{d}c} \Upsilon_k(c) \geq 0$; \\
  $(b)$ $\Upsilon_k$ is a convex function.
\end{proposition}

Using the weight set $\mathcal{W}_\phi^\tau$ defined in \eqref{W_phi^tau}, problem \eqref{Original} is approximated by
\begin{equation}
  \begin{array}{ll}
    \underset{\bm{x} \in \mathcal{X}, \alpha \in \mathbb{R}}{\textrm{Min.}} & \underset{\bm{w} \in \mathcal{W}_{\phi}^{\tau}}{\max} \left\{ \displaystyle \sum_{i=1}^{n} w_{i} \pi^\textrm{c}(\bm{x}; \hat{\bm{\xi}}_i) \right\} \\
    \st & \underset{\bm{w} \in \mathcal{W}_{\phi}^{\tau}}{\max} \left\{ \displaystyle \sum_{i=1}^n w_{i} \Upsilon_{k}( \pi^\textrm{c}(\bm{x}; \hat{\bm{\xi}}_i) - \alpha ) \right\} \leq (1-\gamma)(\nu-\alpha).
  \end{array} \label{DRO_Compliance}
\end{equation}
Here, the objective value of problem \eqref{DRO_Compliance} is transformed through the same process in Section \ref{3.2}.

Through a process similar to that described in Section \ref{2.2}, problem \eqref{DRO_Compliance} can be transformed into a single-stage optimization problem as follows: 
\begin{equation}
\begin{array}{lll}
    & \textrm{Min.} & \tau \lambda_2 + \eta_2 + \lambda_2 \displaystyle \sum_{i=1}^n w_i^0 \phi^* \left( \frac{\iota_{2i}}{\lambda_2} \right) \\
    & \st & \iota_{2i} \geq \pi^\textrm{c}(\bm{x}; \hat{\bm{\xi}}_i) - \eta_2,\quad i=1,\dots,n, \\
    & & \tau \lambda_1 + \eta_1 + \lambda_1 \displaystyle \sum_{i=1}^n w_i^0 \phi^* \left( \frac{\iota_{1i}}{\lambda_1} \right) \leq (1-\gamma)(\nu-\alpha), \\
    & & \iota_{1i} \geq \Upsilon_k (\pi^\textrm{c}(\bm{x}; \hat{\bm{\xi}}_i) - \alpha) - \eta_1,\quad i=1,\dots,n, \\
    & & \bm{x} \in \mathcal{X},\quad \lambda_1 \geq 0,\quad \lambda_2 \geq 0,
\end{array} \label{Extension}
\end{equation}
where the optimization variables are $\bm{x} \in \mathbb{R}^m, \alpha \in \mathbb{R}, \lambda_1 \in \mathbb{R},  \lambda_2 \in \mathbb{R}, \eta_1 \in \mathbb{R}, \eta_2 \in \mathbb{R}, \bm{\iota}_1 \in \mathbb{R}^n,$ and $\bm{\iota}_2 \in \mathbb{R}^n$.

We can demonstrate the following proposition. 
\begin{proposition} \label{prop_convex}
If $\tau>0$ is sufficiently small so that $\mathcal{W}_{\phi}^{\tau} \subseteq \mathbb{R}_{++}^m$ is satisfied and Assumptions \ref{assumption1} and \ref{assumption2} hold, then problem \eqref{Extension} is a convex optimization problem.
\end{proposition}
\begin{proof}
  See Section \ref{A.1}.
\end{proof}

\begin{remark}
Without consideration for uncertainty of the input distribution, distributionally robust CVaR constrained expected value minimization problem for the compliance of trusses \eqref{DRO_Compliance} is expressed as
\begin{equation}
\begin{array}{lll}
    & \underset{\bm{x} \in \mathcal{X}, \alpha \in \mathbb{R}}{\textrm{Min.}} & \displaystyle \sum_{i=1}^n w_i \pi^\textrm{c}(\bm{x}; \hat{\bm{\xi}}_i) \\
    & \st & \alpha + \displaystyle \frac{1}{1-\gamma} \sum_{i=1}^n w_i \Upsilon_k ( \pi^\textrm{c}(\bm{x}; \hat{\bm{\xi}}_i) - \alpha ) \leq \nu.
\end{array} \label{KDE-based approximation}
\end{equation}
\end{remark}

\subsection{Modified $\chi^2$ distance} \label{3.3}
In problem \eqref{Extension}, $\phi$ is the function that is used to 
define the $\phi$-divergence $I_\phi$ in \eqref{phi-divergence}. When 
using the modified $\chi^2$ distance as the $\phi$-divergence, we can derive the specific form of problem \eqref{Extension}.

The function $\phi$ corresponding to the modified $\chi^2$ distance is defined by
\begin{align*}
\phi (t) =
\begin{cases}
  {(t-1)}^2 & \text{if}\ t \geq 0, \\
  \infty & \text{if}\ t < 0.
\end{cases}
\end{align*}
The conjugate function of $\phi$ is given by
\begin{align}
  \phi^* (s) = \frac{1}{4} {({[s+2]}^+)}^2 -1, \label{modiefied_conjugate}
\end{align}
which satisfies the following proposition.
\begin{proposition} \label{nes_suf_condition}
$y, z \in \mathbb{R}$ satisfy $z \geq \phi^* (y)$ if and only if there 
  exists an $a \in \mathbb{R}$ satisfying 
  \begin{align*}
  z+1 
  \geq
  {\left\| 
  \begin{bmatrix}
   z - 1  \\
   a  \\
  \end{bmatrix}
  \right\|},\quad
  a \geq y + 2,\quad 
  a \geq 0.
  \end{align*}
\end{proposition}
\begin{proof}
  See Section \ref{A.2}.
\end{proof}

Application of Proposition \ref{nes_suf_condition} to problem \eqref{Extension} yields 
\begin{equation}
  \begin{array}{llll}
    & \textrm{Min.} & & (\tau-1) \lambda_2 + \eta_2 + \displaystyle \sum_{i=1}^n w_i^0 z_{2i} \\
    & \st & & z_{2i} + \lambda_2 \geq 
    {\left\| 
    \begin{bmatrix}
      z_{2i} - \lambda_2 \\
      y_{2i}
    \end{bmatrix}
    \right\|},\quad i=1,\dots,n, \\
    & & & y_{2i} \geq \iota_{2i} + 2\lambda_2,\quad i=1,\dots,n, \\
    & & & \iota_{2i} \geq \pi^\textrm{c}(\bm{x}; \hat{\bm{\xi}}_i) - \eta_2,\quad i=1,\dots,n, \\
    & & & (\tau - 1) \lambda_1 + \eta_1 + \displaystyle \sum_{i=1}^n w_i^0 z_{1i} \leq (1-\gamma)(\nu-\alpha), \\
    & & & z_{1i} + \lambda_1 \geq 
    {\left\| 
    \begin{bmatrix}
      z_{1i} - \lambda_1 \\
      y_{1i}
    \end{bmatrix}
    \right\|},\quad i=1,\dots,n, \\
    & & & y_{1i} \geq \iota_{1i} + 2\lambda_1,\quad i=1,\dots,n, \\
    & & & \iota_{1i} \geq \Upsilon_k (\pi^\textrm{c}(\bm{x}; \hat{\bm{\xi}}_i) - \alpha ) - \eta_1,\quad i=1,\dots,n, \\
    & & & \bm{x} \in \mathcal{X},\quad \lambda_1 \geq 0,\quad \lambda_2 \geq 0,\quad \bm{y}_1 \geq \bm{0},\quad \bm{y}_2 \geq \bm{0}.
  \end{array} \label{Extension_Modified}
\end{equation}
Here, the optimization variables are $\bm{x} \in \mathbb{R}^m, \alpha \in \mathbb{R}, \lambda_1 \in \mathbb{R}, \lambda_2 \in \mathbb{R}, \eta_1 \in \mathbb{R}, \eta_2 \in \mathbb{R}, \bm{\iota}_1 \in \mathbb{R}^n, \bm{\iota}_2 \in \mathbb{R}^n, \bm{y}_1 \in \mathbb{R}^n, \bm{y}_2 \in \mathbb{R}^n, \bm{z}_1 \in \mathbb{R}^n$, and $\bm{z}_2 \in \mathbb{R}^n$.

\section{Second-order cone programming formulation with uniform kernel} \label{4}
Problem \eqref{Extension} involves $\Upsilon_k$ defined by 
\eqref{Upsilon_decomp}, and is thereby difficult to solve directly. 
In this section and Section \ref{5}, we transform problem \eqref{Extension_Modified} into a more tractable form that does not involve $\Upsilon_k$.

In this section, we consider $\Upsilon_k$ corresponding to the uniform 
kernel, and show that the value of $\Upsilon_k$ can be expressed as the 
optimal value of a convex optimization problem. We then recast problem 
\eqref{Extension_Modified} into a tractable form.
The uniform kernel is defined as
\begin{align}
  k(y) =
\begin{cases}
  \frac{1}{2} & \textrm{if}\quad |y| \leq 1, \\
  0 & \textrm{otherwise}.
\end{cases} \label{Uniform_kernel}
\end{align}
When using this kernel, the following proposition shows that $\Upsilon_k(c)$ can be expressed as the optimal value of an optimization problem.
\begin{proposition} \label{Uniform_decomp}
  Let $\Upsilon_k$ be defined by \eqref{Upsilon_decomp} with 
  $k$ in \eqref{Uniform_kernel} and $h\in\mathbb{R}_{++}$. Then, for any $c \in \mathbb{R}$, 
  $\Upsilon_k(c)$ is equal to the optimal value of the following second-order cone programming problem:
\begin{equation}
\begin{array}{ll}
    \mathrm{Min.} & c_\rma + s \\
    \st & s + h \geq
    {\left\| 
    \begin{bmatrix}
        s - h \\
        c_\rmq
    \end{bmatrix}
    \right\|}, \\
    & c_\rma + c_\rmq \geq c + h, \\
    & c_\rma \geq 0,\quad 0 \leq c_\rmq \leq 2h,
\end{array} \label{Uniform Upsilon}
\end{equation}
where the optimization variables are $c_\rma \in \mathbb{R}$, $c_\rmq \in \mathbb{R}$, and $s \in \mathbb{R}$.
\end{proposition}
\begin{proof}
  See Section \ref{A.3}.
\end{proof}

Using Proposition \ref{Compliance_can_be_SOCP} and \ref{Uniform_decomp}, 
 we can transform problem \eqref{Extension_Modified} into the following 
 second-order cone programming problem: 
\begin{equation}
\begin{array}{ll}
    \textrm{Min.} & (\tau - 1) \lambda_2 + \eta_2 + \displaystyle \sum_{i=1}^n w_i^0 z_{2i} \\
    \st & z_{2i} + \lambda_2 \geq 
    {\left\| 
    \begin{bmatrix}
        z_{2i} - \lambda_2 \\
        y_{2i}
    \end{bmatrix}
    \right\|},\quad i=1,\dots,n, \\
    & y_{2i} \geq \iota_{2i} + 2\lambda_2,\quad i=1,\dots,n, \\
    & \iota_{2i} \geq \displaystyle \sum_{j=1}^m 2b_{ij} - \eta_2,\quad i=1,\dots,n, \\
    & (\tau - 1) \lambda_1 + \eta_1 + \displaystyle \sum_{i=1}^n w_i^0 z_{1i} \leq (1-\gamma)(\nu-\alpha), \\
    & z_{1i} + \lambda_1 \geq 
    {\left\| 
    \begin{bmatrix}
        z_{1i} - \lambda_1 \\
        y_{1i}
    \end{bmatrix}
    \right\|},\quad i=1,\dots,n, \\
    & y_{1i} \geq \iota_{1i} + 2\lambda_1,\quad i=1,\dots,n, \\
    & \iota_{1i} \geq c_{\rma i} + s_i - \eta_1,\quad i=1,\dots,n, \\
    & s_i + h \geq
    {\left\| 
    \begin{bmatrix}
        s_i - h \\
        c_{\rmq i}
    \end{bmatrix}
    \right\|},\quad i=1,\dots,n, \\
    & c_{\rma i} + c_{\rmq i} \geq \displaystyle \sum_{j=1}^m 2b_{ij} - \alpha + h,\quad i=1,\dots,n, \\
    & b_{ij} + x_j \geq 
    {\left\| 
    \begin{bmatrix}
        b_{ij} - x_j \\
        \sqrt{2l_{j}/E} q_{ij}
    \end{bmatrix}
    \right\|},\quad i=1,\dots,n,\ j=1,\dots,m, \\
    & \displaystyle \sum_{j=1}^m q_{ij} \bm{\beta}_{j} = \hat{\bm{\xi}}_i,\quad i=1,\dots,n, \\
    & \bm{x} \in \mathcal{X},\quad \lambda_1 \geq 0,\quad \lambda_2 \geq 0,\quad \bm{y}_1 \geq \bm{0},\quad \bm{y}_2 \geq \bm{0}, \\
    & \bm{0} \leq \bm{c}_\rmq \leq 2h \bm{1},\quad \bm{c}_\rma \geq \bm{0},
\end{array} \label{Extension_Modified_Uniform_Compliance}
\end{equation}
where the optimization variables are $\bm{x} \in \mathbb{R}^m$, $\alpha \in \mathbb{R}$, $\lambda_1 \in \mathbb{R}$, $\lambda_2 \in \mathbb{R}$, $\eta_1 \in \mathbb{R}$, $\eta_2 \in \mathbb{R}$, $\bm{\iota}_1 \in \mathbb{R}^n$, $\bm{\iota}_2 \in \mathbb{R}^n$, $\bm{y}_1 \in \mathbb{R}^n$, $\bm{y}_2 \in \mathbb{R}^n$, $\bm{z}_1 \in \mathbb{R}^n$, $\bm{z}_2 \in \mathbb{R}^n$, $\bm{c}_\rmq \in \mathbb{R}^n$, $\bm{c}_\rma \in \mathbb{R}^n$, $\bm{s} \in \mathbb{R}^n$, $\bm{b}_1, \dots, \bm{b}_n \in \mathbb{R}^m$, and $\bm{q}_1, \dots, \bm{q}_n \in \mathbb{R}^m$. 
It is worth noting that this problem can be solved more efficiently 
with a primal-dual interior point method \cite{AL12}.

\begin{remark}
When $\tau=0$, problem \eqref{Extension_Modified} can be expressed as the 
  second-order cone programming problem:
\begin{equation}
\begin{array}{ll}
    \mathrm{Min.} & \displaystyle \sum_{i=1}^n w_i^0  \pi^\mathrm{c}(\bm{x}; \hat{\bm{\xi}}_i) \\
    \st & \displaystyle \sum_{i=1}^n w_i^0 \left( c_{\rma i} + s_i \right) \leq (1-\gamma) (\nu - \alpha), \\
    & s_{i} + h \geq 
    {\left\| 
    \begin{bmatrix}
        s_{i} - h \\
        c_{\rmq i}
    \end{bmatrix}
    \right\|},\quad i=1,\dots,n, \\
    & c_{\rma i} + c_{\rmq i} \geq \displaystyle \sum_{j=1}^m 2b_{ij} - \alpha + h,\quad i=1,\dots,n, \\
    & b_{ij} + x_j \geq 
    {\left\| 
    \begin{bmatrix}
        b_{ij} - x_j \\
        \sqrt{2l_{j}/E} q_{ij}
    \end{bmatrix}
    \right\|},\quad i=1,\dots,n,\ j=1,\dots,m, \\
    & \displaystyle \sum_{j=1}^m q_{ij} \bm{\beta}_{j} = \hat{\bm{\xi}}_i,\quad i=1,\dots,n, \\
    & \bm{x} \in \mathcal{X},\quad \bm{c}_\rma \geq \bm{0},\quad \bm{0} \leq \bm{c}_\rmq \leq 2h\bm{1},
\end{array} \label{Extension_tau0_Uniform_Compliance}
\end{equation}
where the optimization variables are $\bm{x} \in \mathbb{R}^m$, $\alpha \in \mathbb{R}$, $\bm{c}_\rma \in \mathbb{R}^n$, $\bm{c}_\rmq \in \mathbb{R}^n$, $\bm{s} \in \mathbb{R}^n$, $\bm{b}_1 \dots, \bm{b}_n \in \mathbb{R}^m$, and $\bm{q}_1, \dots, \bm{q}_n \in \mathbb{R}^m$.
\end{remark}

\section{Second-order cone programming formulation with triangular kernel} \label{5}
In this section, we consider $\Upsilon_k$ corresponding to the 
triangular kernel, and express its value as the optimal value of a 
convex optimization problem. We then recast problem 
\eqref{Extension_Modified} into a tractable form.

The triangular kernel is defined by
\begin{align}
  k(y) =
\begin{cases}
  y+1 & \textrm{if}\quad {-1} \leq y \leq 0, \\
  y-1 & \textrm{if}\quad 0 \leq y \leq 1, \\
  0 & \textrm{otherwise}.
\end{cases} \label{Triangular_kernel}
\end{align} 
\begin{proposition} \label{Triangular_decomp}
  Let $\Upsilon_k$ be defined by \eqref{Upsilon_decomp} with 
  $k$ in \eqref{Triangular_kernel} and $h\in\mathbb{R}_{++}$. 
  Then, for any $c \in \mathbb{R}$, $\Upsilon_k(c)$ 
  is equal to the optimal value of the second-order cone programming problem:
\begin{equation}
\begin{array}{ll}
    \mathrm{Min.} & c_\rma + \displaystyle \frac{s_1 + s_2}{6h^2} + s_3 \\
    \st & s_1 + c_{\mathrm{c1}} \geq
    {\left\| 
    \begin{bmatrix}
        s_1 - c_{\mathrm{c1}} \\
        2r_1
    \end{bmatrix}
    \right\|},\quad 
    r_1 + \displaystyle \frac{1}{4} \geq
    {\left\| 
    \begin{bmatrix}
        r_1 - \frac{1}{4} \\
        c_{\mathrm{c1}}
    \end{bmatrix}
    \right\|}, \\ \\
    & s_2 + c_{\mathrm{c2}} \geq
    {\left\| 
    \begin{bmatrix}
        s_2 - c_{\mathrm{c2}} \\
        2r_2
    \end{bmatrix}
    \right\|},\quad 
    r_2 + \displaystyle \frac{1}{4} \geq
    {\left\| 
    \begin{bmatrix}
        r_2 - \frac{1}{4} \\
        c_{\mathrm{c2}}
    \end{bmatrix}
    \right\|}, \\ \\
    & s_3 + c_{\mathrm{c2}} \geq \displaystyle \frac{5h}{6}, \\
    & c_\rma + c_{\mathrm{c1}} - c_{\mathrm{c2}} \geq c, \\
    & 0 \leq c_{\mathrm{c1}} \leq h,\quad 0 \leq c_{\mathrm{c2}} \leq h,\quad c_\rma \geq 0,
\end{array} \label{Triangular Upsilon}
\end{equation}
where the optimization variables are $c_{\mathrm{c1}} \in \mathbb{R}$, 
  $c_{\mathrm{c2}} \in \mathbb{R}$, $c_\rma \in \mathbb{R}$, $s_1 \in 
  \mathbb{R}$, $s_2 \in \mathbb{R}$, $s_3 \in \mathbb{R}$, $r_1 \in 
  \mathbb{R}$, and $r_2 \in \mathbb{R}$.
\end{proposition}
\begin{proof}
  See Section \ref{A.4}.
\end{proof}
Using Propositions \ref{Compliance_can_be_SOCP} and \ref{Triangular_decomp}, we can transform problem \eqref{Extension_Modified} into the following second-order cone programming problem:
\begin{equation}
  \begin{array}{ll}
    \mathrm{Min.} & (\tau - 1) \lambda_2 + \eta_2 + \displaystyle \sum_{i=1}^n w_i^0 z_{2i} \\
    \st & z_{2i} + \lambda_2 \geq 
    {\left\| 
    \begin{bmatrix}
        z_{2i} - \lambda_2 \\
        y_{2i}
    \end{bmatrix}
    \right\|},\quad i=1,\dots,n, \\
    & y_{2i} \geq \iota_{2i} + 2\lambda_2,\quad i=1,\dots,n, \\
    & \iota_{2i} \geq \displaystyle \sum_{j=1}^m 2b_{ij} - \eta_2,\quad i=1,\dots,n, \\
    & (\tau - 1) \lambda_1 + \eta_1 + \displaystyle \sum_{i=1}^n w_i^0 z_{1i} \leq (1-\gamma)(\nu-\alpha), \\
    & z_{1i} + \lambda_1 \geq 
    {\left\| 
    \begin{bmatrix}
        z_{1i} - \lambda_1 \\
        y_{1i}
    \end{bmatrix}
    \right\|},\quad i=1,\dots,n, \\
    & y_{1i} \geq \iota_{1i} + 2\lambda_1,\quad i=1,\dots,n, \\
    & \iota_{1i} \geq c_{\rma i} + \displaystyle \frac{s_{1i} + s_{2i}}{6h^2} + s_{3i} - \eta_1,\quad i=1,\dots,n, \\
    \\
    & s_{1i} + c_{\mathrm{c1} i} \geq
    {\left\| 
    \begin{bmatrix}
        s_{1i} - c_{\mathrm{c1} i} \\
        2r_{1i}
    \end{bmatrix}
    \right\|},\quad 
    r_{1i} + \displaystyle \frac{1}{4} \geq
    {\left\| 
    \begin{bmatrix}
        r_{1i} - \frac{1}{4} \\
        c_{\mathrm{c1} i}
    \end{bmatrix}
    \right\|},\quad i=1,\dots,n, \\
    & s_{2i} + c_{\mathrm{c2} i} \geq
    {\left\| 
    \begin{bmatrix}
        s_{2i} - c_{\mathrm{c2} i} \\
        2r_{2i}
    \end{bmatrix}
    \right\|},\quad 
    r_{2i} + \displaystyle \frac{1}{4} \geq
    {\left\| 
    \begin{bmatrix}
        r_{2i} - \frac{1}{4} \\
        c_{\mathrm{c2} i}
    \end{bmatrix}
    \right\|},\quad i=1,\dots,n, \\
    \\
    & s_{3i} + c_{\mathrm{c2} i} \geq \displaystyle \frac{5h}{6},\quad i=1,\dots,n, \\
    & c_{\rma i} + c_{\mathrm{c1} i} - c_{\mathrm{c2} i} \geq \displaystyle \sum_{j=1}^m 2b_{ij} - \alpha,\quad i=1,\dots,n, \\
    & b_{ij} + x_j \geq 
    {\left\| 
    \begin{bmatrix}
        b_{ij} - x_j \\
        \sqrt{2l_{j}/E} q_{ij}
    \end{bmatrix}
    \right\|},\quad i=1,\dots,n,\ j=1,\dots,m, \\
    & \displaystyle \sum_{j=1}^m q_{ij} \bm{\beta}_j = \hat{\bm{\xi}}_i,\quad i=1,\dots,n, \\
    & \bm{x} \in \mathcal{X},\quad \lambda_1 \geq 0,\quad \lambda_2 \geq 0,\quad \bm{y}_1 \geq \bm{0},\quad \bm{y}_2 \geq \bm{0}, \\
    & \bm{0} \leq \bm{c}_\mathrm{c1} \leq h \bm{1},\quad \bm{0} \leq \bm{c}_\mathrm{c2} \leq h \bm{1},\quad \bm{c}_\rma \geq \bm{0},
  \end{array} \label{Extension_Modified_Triangular_Compliance}
\end{equation}
where the optimization variables are $\bm{x} \in \mathbb{R}^m$, $\alpha \in \mathbb{R}$, $\lambda_1 \in \mathbb{R}$, $\lambda_2 \in \mathbb{R}$, $\eta_1 \in \mathbb{R}$, $\eta_2 \in \mathbb{R}$, $\bm{\iota}_1 \in \mathbb{R}^n$, $\bm{\iota}_2 \in \mathbb{R}^n$, $\bm{y}_1 \in \mathbb{R}^n$, $\bm{y}_2 \in \mathbb{R}^n$, $\bm{z}_1 \in \mathbb{R}^n$, $\bm{z}_2 \in \mathbb{R}^n$, $\bm{c}_\mathrm{c1} \in \mathbb{R}^n$, $\bm{c}_\mathrm{c2} \in \mathbb{R}^n$, $\bm{c}_\rma \in \mathbb{R}^n$, $\bm{s}_1 \in \mathbb{R}^n$, $\bm{s}_2 \in \mathbb{R}^n$, $\bm{s}_3 \in \mathbb{R}^n$, $\bm{r}_1 \in \mathbb{R}^n$, $\bm{r}_2 \in \mathbb{R}^n$, $\bm{b}_1 \dots, \bm{b}_n \in \mathbb{R}^m$, and $\bm{q}_1, \dots, \bm{q}_n \in \mathbb{R}^m$.

\begin{remark}
When $\tau=0$, problem \eqref{DRO_Compliance} can be expressed as the 
  following second-order cone programming problem:
\begin{equation}
\begin{array}{ll}
    \mathrm{Min.} & \displaystyle \sum_{i=1}^n w_i^0  \pi^\mathrm{c}(\bm{x}; \hat{\bm{\xi}}_i) \\
    \st & \displaystyle \sum_{i=1}^n w_i^0 \left( c_{\rma i} + \frac{s_{1i} + s_{2i}}{6h^2} \right) \leq (1-\gamma) (\nu - \alpha), \\
    & s_{1i} + c_{\mathrm{c1}i} \geq
    {\left\| 
    \begin{bmatrix}
        s_{1i} - c_{\mathrm{c1}i} \\
        2r_{1i}
    \end{bmatrix}
    \right\|},\quad 
    r_{1i} + \displaystyle \frac{1}{4} \geq
    {\left\| 
    \begin{bmatrix}
        r_{1i} - \frac{1}{4} \\
        c_{\mathrm{c1}i}
    \end{bmatrix}
    \right\|},\quad i=1,\dots,n, \\
    \\
    & s_{2i} + c_{\mathrm{c2}i} \geq
    {\left\| 
    \begin{bmatrix}
        s_{2i} - c_{\mathrm{c2}i} \\
        2r_{2i}
    \end{bmatrix}
    \right\|},\quad
    r_{2i} + \displaystyle \frac{1}{4} \geq
    {\left\| 
    \begin{bmatrix}
        r_{2i} - \frac{1}{4} \\
        c_{\mathrm{c2}i}
    \end{bmatrix}
    \right\|},\quad i=1,\dots,n, \\
    \\
    & s_{3i} + c_{\mathrm{c2}i} \geq \displaystyle \frac{5h}{6},\quad i=1,\dots,n, \\
    & c_{\rma i} + c_{\mathrm{c1}i} - c_{\mathrm{c2}i} \geq \displaystyle \sum_{j=1}^m 2b_{ij} - \alpha,\quad i=1,\dots,n, \\
    & b_{ij} + x_j \geq 
    {\left\| 
    \begin{bmatrix}
        b_{ij} - x_j \\
        \sqrt{2l_{j}/E} q_{ij}
    \end{bmatrix}
    \right\|},\quad i=1,\dots,n,\ j=1,\dots,m, \\
    & \displaystyle \sum_{j=1}^m q_{ij} \bm{\beta}_j = \hat{\bm{\xi}}_i,\quad i=1,\dots,n, \\
    & \bm{x} \in \mathcal{X},\quad \bm{0} \leq \bm{c}_{\mathrm{c1}} \leq h \bm{1},\quad \bm{0} \leq \bm{c}_{\mathrm{c2}} \leq h \bm{1},\quad \bm{c}_\rma \geq \bm{0},
\end{array} \label{Extension_tau0_Triangular_Compliance}
\end{equation}
where the optimization problems are $\bm{x} \in \mathbb{R}^m$, $\alpha 
  \in \mathbb{R}$, $\bm{c}_{\mathrm{c1}} \in \mathbb{R}^n$, 
  $\bm{c}_{\mathrm{c2}} \in \mathbb{R}^n$, $\bm{c}_\rma \in 
  \mathbb{R}^n$, $\bm{s}_1 \in \mathbb{R}^n$, $\bm{s}_2 \in 
  \mathbb{R}^n$, $\bm{s}_3 \in \mathbb{R}^n$, $\bm{r}_1 \in 
  \mathbb{R}^n$, $\bm{r}_2 \in \mathbb{R}^n$, $\bm{b}_1 \dots, \bm{b}_n 
  \in \mathbb{R}^m$, and $\bm{q}_1, \dots, \bm{q}_n \in \mathbb{R}^m$.
\end{remark}

\section{Numerical experiments} \label{6}
In this section, we conduct numerical experiments on problem \eqref{Extension_Modified_Uniform_Compliance} and problem \eqref{Extension_Modified_Triangular_Compliance}, which are 
distributionally robust CVaR-constrained expected value minimization 
problem.
The problems were solved using MATLAB R2023b with CVX version 2.1 \cite{CVX} and the solver MOSEK ver.~9.1.9 \cite{mosek}. All computations were performed on a PC running Windows 11 (Intel Core i7, 1.8{\,}GHz CPU, 16{\,}GB RAM).

\subsection{2-bar truss} \label{6.1}
Consider the truss shown in Figure \ref{fig:2-bar_truss}. The truss has $m=2$ members and $d=2$ degrees of freedom of the nodal displacements. 
In  Figure \ref{fig:2-bar_truss}, the filled circles represent fixed nodes and the open circle represents a free node.
The length of the horizontal member is $1\,\mathrm{m}$. 
The Young modulus of the members and the upper bound of the total structural volume are 
$E=20\,\mathrm{GPa}$ and $\overline{V}=1000\,\mathrm{mm}^3$, respectively.

\begin{figure}[!htbp]
  \centering
  \includegraphics[width=30mm]{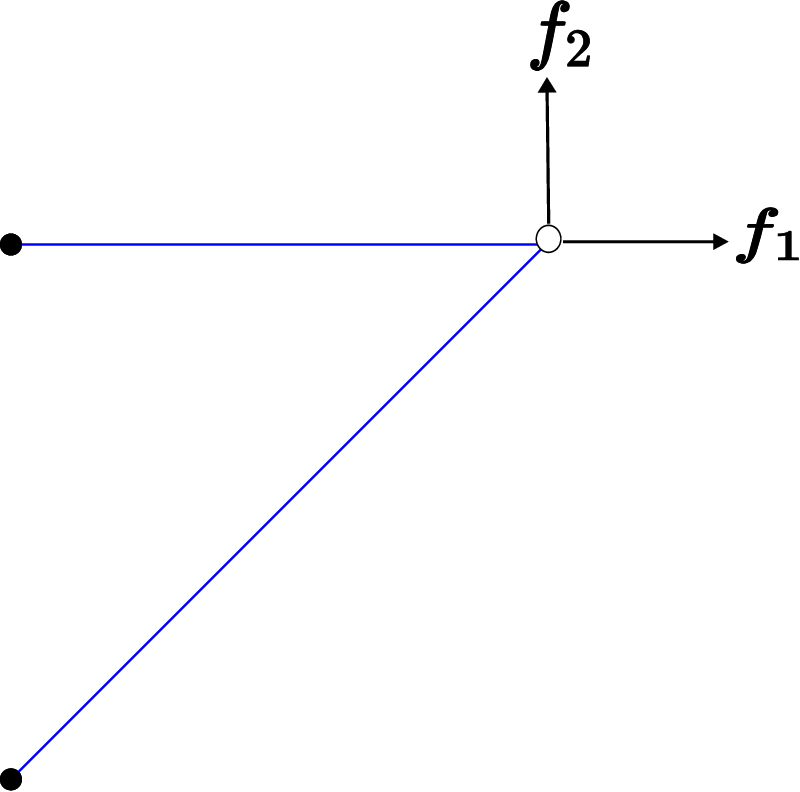}
  \caption{2-bar truss}
  \label{fig:2-bar_truss}
\end{figure}

We obtain the Pareto optimal solutions of the bi-objective minimization of the 
worst-case expected value and the worst-case CVaR of the compliance, 
and analyze their trade-off relation. Figure \ref{fig:scatter plot of 
generated data2} shows $n=50$ samples of external forces $\{\hat{\bm{\xi}}_1, \dots, \hat{\bm{\xi}}_n\}$ used in the numerical experiments of this section. In the figure, the horizontal and vertical components of the external forces are represented by $f_1$ and $f_2$, respectively, and are shown along the $x$-axis and $y$-axis. These samples are drawn from a bivariate normal 
distribution $N(\bm{\mu}_1, \varSigma_1)$ with mean vector $\bm{\mu}_1$ 
and variance-covariance matrix $\varSigma_1$, where 
\begin{equation*}
    \bm{\mu}_1 = 
    \begin{bmatrix}
        100 \\
        0
    \end{bmatrix}\,\mathrm{kN}, \quad
    \varSigma_1 = 
    \begin{bmatrix}
        150 & 50 \\
        50 & 100
    \end{bmatrix}\,\mathrm{kN}^2.
\end{equation*}
In Figure \ref{fig:scatter plot of generated data2}, each open circle depicts a sample, while the filled square represents the mean vector $\bm{\mu}_1$. 
\begin{figure}
\centering
\includegraphics[width=60mm]{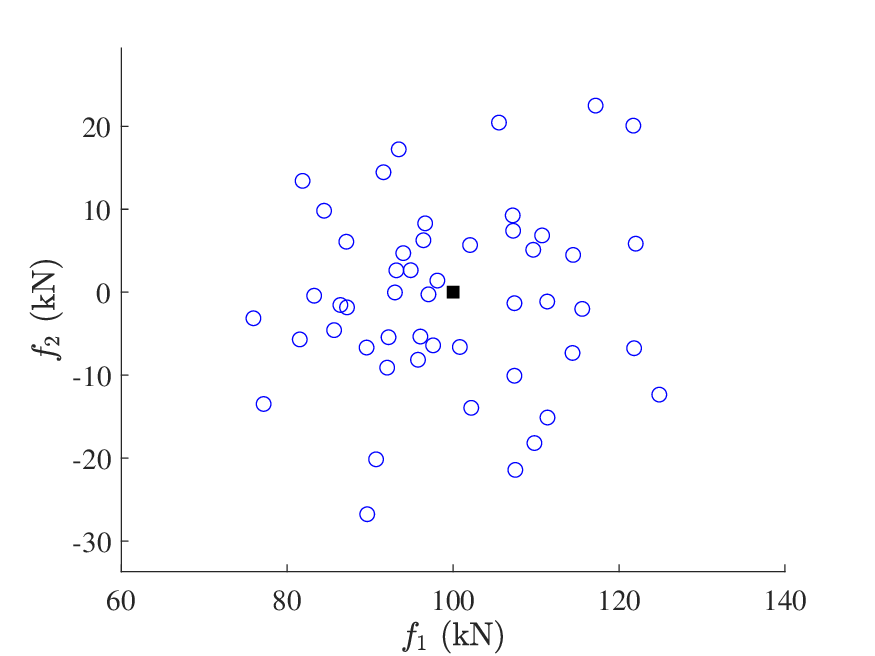}
\caption{50 samples used in Section~6.1}
\label{fig:scatter plot of generated data2}
\end{figure}
The sample mean vector $\hat{\bm{\mu}}_{1}$ and the sample variance-covariance matrix $\hat{\varSigma}_{1}$ of $\{ \hat{\bm{\xi}}_1,\dots,\hat{\bm{\xi}}_n \}$ are
\begin{equation*}
    \hat{\bm{\mu}}_{1} = 
    \begin{bmatrix}
        99.491 \\
        -0.810
    \end{bmatrix}\,\mathrm{kN}, \quad
    \hat{\varSigma}_{1} = 
    \begin{bmatrix}
        156.15 & 12.92 \\
        12.92 & 119.21
    \end{bmatrix}\,\mathrm{kN}^2.
\end{equation*}
Parameters in problems \eqref{Extension_Modified_Uniform_Compliance} are set as 
\begin{equation}
 \gamma = 0.95,\quad h = 10\,\mathrm{J},\quad \bm{w}^0 = \frac{1}{50} \bm{1},\quad \tau = 0.3. \label{setting3}
\end{equation}

We obtain the Pareto front 
by solving problem \eqref{Extension_Modified_Uniform_Compliance} for different values of the upper bound on the worst-case CVaR, $\nu$.
Figure \ref{fig:Pareto_data2} shows the Pareto front of the worst-case 
expected value and the worst-case CVaR of the compliance, while Figure 
\ref{fig:Shift_optimal_data2} depicts the corresponding Pareto optimal solutions. In Figure \ref{fig:Shift_optimal_data2}, $x_1^*$ and $ x_2^*$ represent the optimal cross-sectional areas of the upper member and the lower member, respectively.
In Figures \ref{fig:Pareto_data2} and \ref{fig:Shift_optimal_data2}, the 
leftmost point corresponds to the solution with the minimum worst-case CVaR. 
In contrast, the rightmost point corresponds to the solution with the minimum worst-case expected value. 
It can be observed from Figure \ref{fig:Shift_optimal_data2} that 
as the worst-case CVaR decreases, the cross-sectional area of the upper member decreases, while the cross-sectional area of the lower member increases to enhance robustness against vertical variations in the external forces.
\begin{figure}
\centering
\begin{subfigure}[t]{0.45\textwidth}
  \includegraphics[width=60mm]{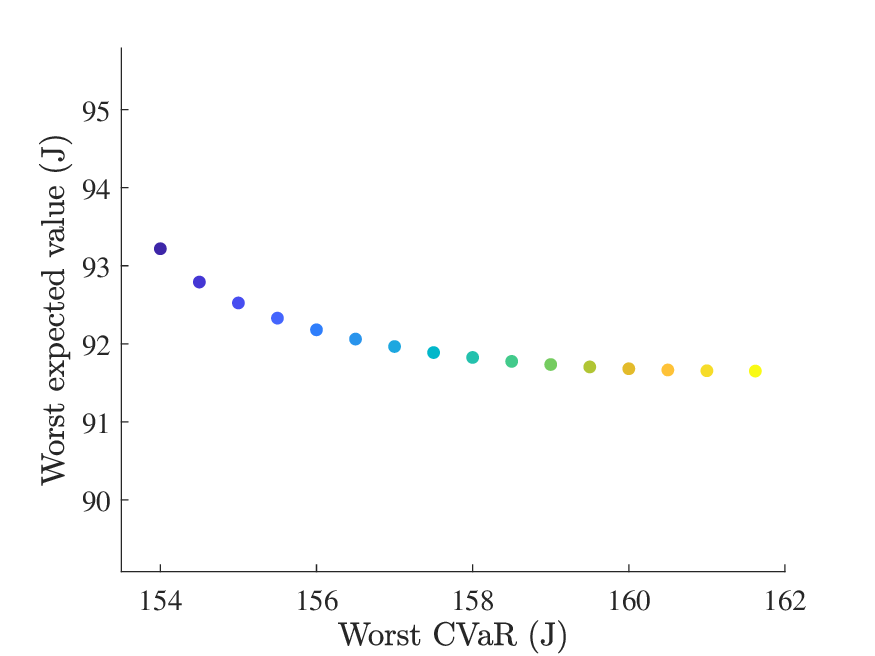}
  \caption{Pareto front.}
  \label{fig:Pareto_data2}
\end{subfigure}
\hfill
\begin{subfigure}[t]{0.45\textwidth}
\centering
\includegraphics[width=60mm]{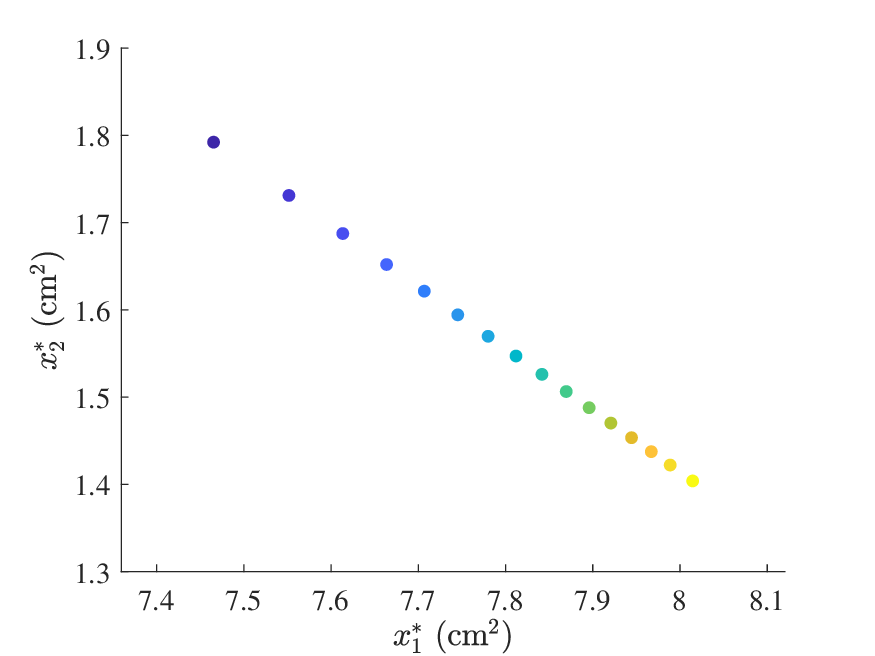}
\caption{Pareto solutions.}
\label{fig:Shift_optimal_data2}
\end{subfigure}
\caption{Pareto front and Pareto solutions}
\end{figure}
Moreover, we investigate the variation of the Pareto front with respect 
to $\tau$, which represents the magnitude of uncertainty. Figure \ref{fig:Pareto_tau0.3-0.5} shows the Pareto fronts for 
$\tau=0.3$, $0.4$, and $0.5$. From Figure \ref{fig:Pareto_tau0.3-0.5}, 
we can observe that as the value of $\tau$ decreases, the feasible 
region of the design variables expands. Consequently, the Pareto front shifts downward as optimal solutions with smaller worst-case expected value of the compliance are obtained. 
\begin{figure}
\centering
\includegraphics[width=60mm]{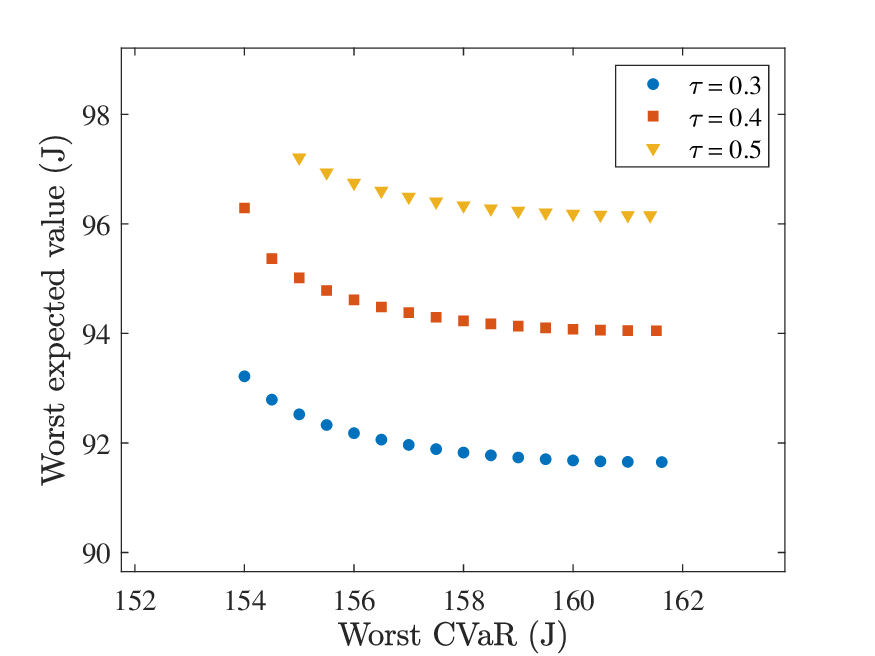}
\caption{Pareto fronts corresponding to different values of $\tau$}
\label{fig:Pareto_tau0.3-0.5}
\end{figure}

\subsection{289-bar truss} \label{6.2}
Consider the planar truss shown in Figure \ref{fig:289-bar_truss}. The truss consists of $m=289$ members, and $d=56$ degrees of freedom for nodal displacements. In Figure \ref{fig:289-bar_truss}, the filled circles represent fixed nodes, while the open circles represent free nodes. The distance between the nearest nodes is $1\,\mathrm{m}$. The Young modulus of the members and the upper bound of the total member volume are $E=20\,\mathrm{GPa}$ and $\overline{V}=20000\,\mathrm{mm}^3$, respectively.
\begin{figure}
\centering
\includegraphics[width=60mm]{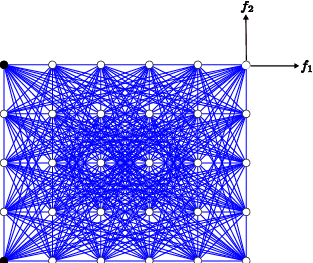}
\caption{289-bar truss in Section~6.2}
\label{fig:289-bar_truss}
\end{figure}
For the numerical experiments in this section, the external forces are assumed to act only on the top-right node of the truss depicted in Figure \ref{fig:289-bar_truss}. The external load follows a mixture distribution composed of two normal distributions. A total of $n=50$ samples were utilized for the experiment. These samples are drawn from a mixture distribution consisting of $N(\bm{\mu}_2, \varSigma_2)$ and $N(\bm{\mu}_3, \varSigma_3)$, where
\begin{alignat*}{3}
    \bm{\mu}_2 &= 
    \begin{bmatrix}
        90 \\
        10
    \end{bmatrix}\,\mathrm{kN}, &\quad
    &\varSigma_2 = 
    \begin{bmatrix}
        100 & 0 \\
        0 & 150
    \end{bmatrix}\,\mathrm{kN}^2, \\
    \bm{\mu}_3 &= 
    \begin{bmatrix}
        -10 \\
        40
    \end{bmatrix}\,\mathrm{kN}, &\quad
    &\varSigma_3 = 
    \begin{bmatrix}
        100 & 0 \\
        0 & 150
    \end{bmatrix}\,\mathrm{kN}^2.
\end{alignat*}
The sample mean vector $\hat{\bm{\mu}}_{2}$ and the sample variance-covariance matrix $\hat{\varSigma}_{2}$ obtained from the samples drawn from distribution $N(\bm{\mu}_2, \varSigma_2)$, as well as the sample mean vector $\hat{\bm{\mu}}_{3}$ and the sample variance-covariance matrix $\hat{\varSigma}_{3}$ obtained from the samples drawn from distribution $N(\bm{\mu}_3, \varSigma_3)$, are as follows:
\begin{alignat*}{3}
    \hat{\bm{\mu}}_{2} &= 
    \begin{bmatrix}
        89.767 \\
        11.054
    \end{bmatrix}\,\mathrm{kN}, &\quad
    &\hat{\varSigma}_{2} = 
    \begin{bmatrix}
        73.42 & 21.04 \\
        21.04 & 107.15
    \end{bmatrix}\,\mathrm{kN}^2, \\
    \hat{\bm{\mu}}_{3} &= 
    \begin{bmatrix}
        -5.422 \\
        35.281
    \end{bmatrix}\,\mathrm{kN}, &\quad
    &\hat{\varSigma}_{3} = 
    \begin{bmatrix}
        46.63 & -14.88 \\
        -14.88 & 107.15
    \end{bmatrix}\,\mathrm{kN}^2.
\end{alignat*}

In this section, we perform experiments using the samples following a mixture distribution. In existing literature, handling data from mixture distributions is a challenge as such data are often modeled as following a specific distribution. This section demonstrates the effectiveness of our approach in addressing the challenge.

Figure \ref{fig:generated_load_data_289-bar_Truss} illustrates the samples used in the numerical experiments. Each ``$\circ$'' and ``$\times$'' correspond to the sample drawn from the normal distribution $N(\bm{\mu}_2, \varSigma_2) $ and the sample from $N(\bm{\mu}_3, \varSigma_3)$, respectively. The filled square and the filled triangle represent mean vectors of these distributions.
\begin{figure}
\centering
\includegraphics[width=60mm]{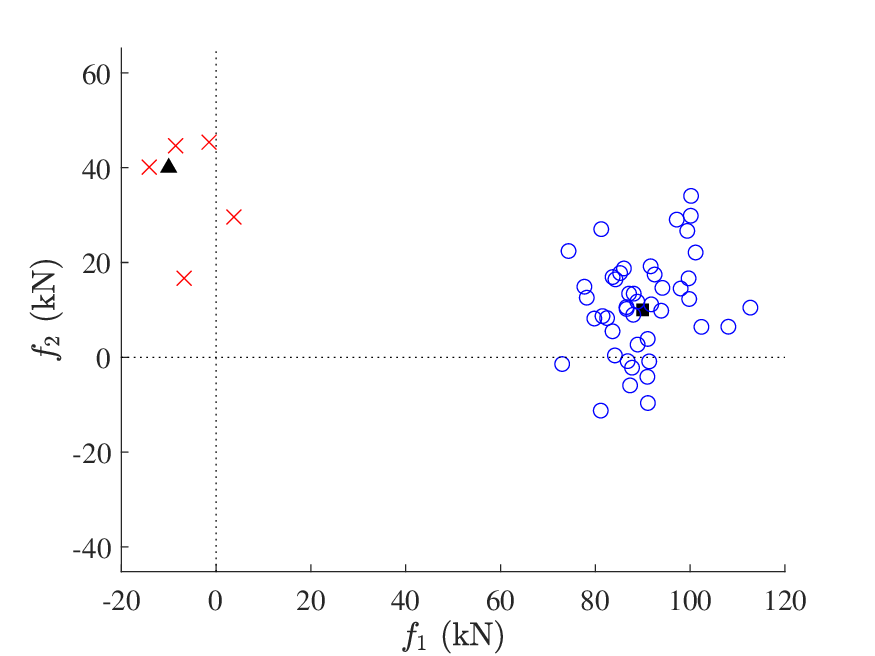}
\caption{50 samples used in Section~6.2}
\label{fig:generated_load_data_289-bar_Truss}
\end{figure}

By solving the distributionally robust CVaR-constrained expected value minimization problem \eqref{Extension_Modified_Uniform_Compliance} using a uniform kernel, the Pareto front of the worst-case expected value and the worst-case CVaR of the compliance was obtained. The settings are as follows:
\begin{equation*}
 \gamma = 0.95,\quad h = 30\,\mathrm{J},\quad \bm{w}^0 = \frac{1}{50} \bm{1}.
\end{equation*}

Figure \ref{fig:Pareto_tau0.3-0.5_289bar_truss} shows the Pareto fronts for $\tau=0.3$, $0.4$, and $0.5$. In the figure, the leftmost point on each Pareto front represents the solution with the minimum worst-case CVaR of the compliance, while the rightmost point represents the solution with the minimum worst-case expected value of the compliance. As the parameter $\tau$, representing the magnitude of uncertainty, decreases, the corresponding distributionally uncertainty set shrinks. Consequently, the Pareto front shifts downward.

\begin{figure}
\centering
\includegraphics[width=60mm]{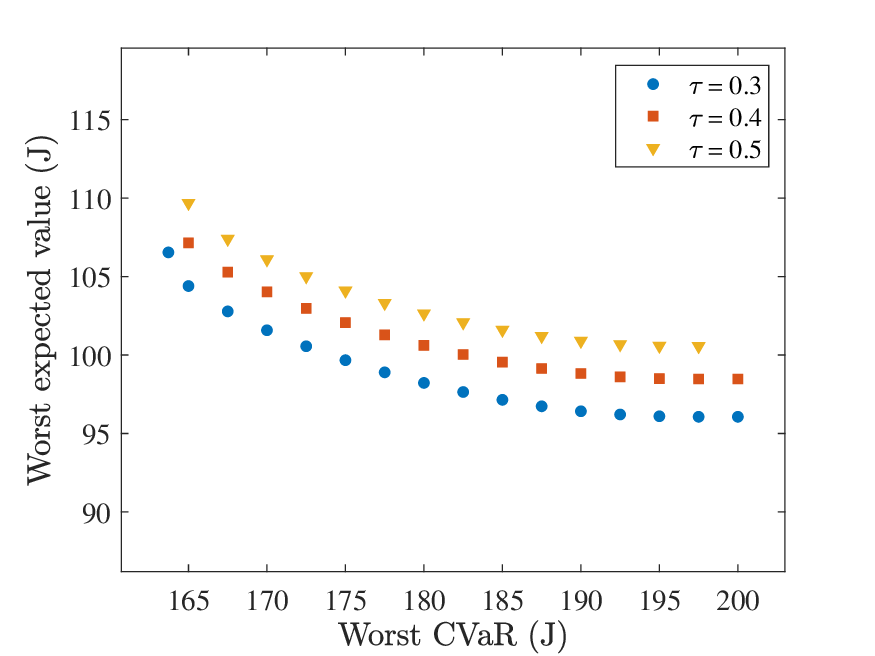}
\caption{Pareto fronts for $\tau=0.3$, $0.4$, and $0.5$}
\label{fig:Pareto_tau0.3-0.5_289bar_truss}
\end{figure}

Figure \ref{fig:Optimal_solutions_289-bar_Truss} illustrates how the optimal solution changes for the optimization problem with $\tau=0.3$. From Figure \ref{fig:Optimal_solutions_289-bar_Truss}, it can be observed that as the upper bound $\nu$ of the worst-case CVaR decreases, the volume of the member between the bottom-left node and the top-right node decreases, while the volume of other members increases to enhance robustness against vertical variations in external force.

\begin{figure}
\centering
\begin{subfigure}[t]{0.45\textwidth}
  \includegraphics[width=60mm]{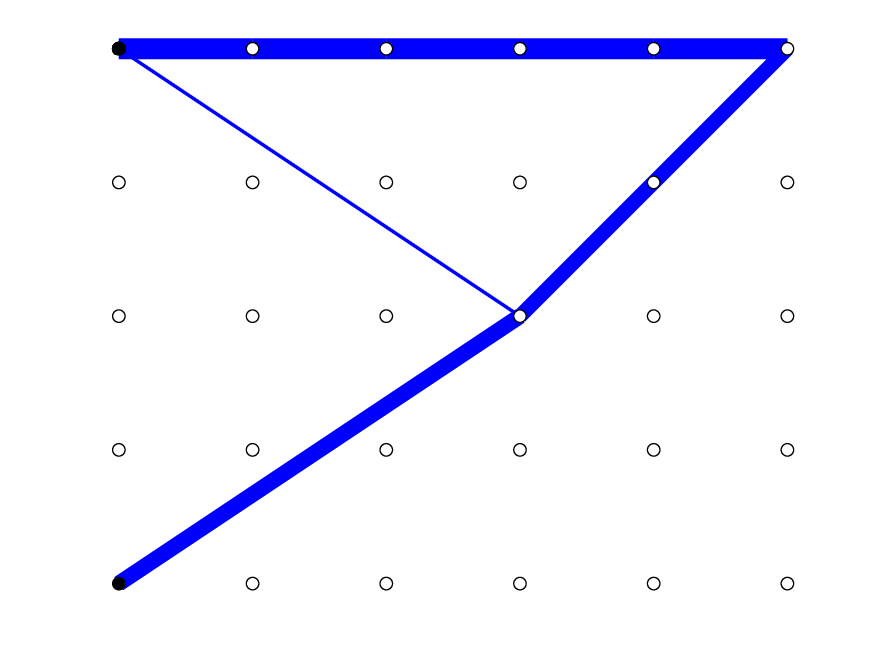}
  \caption{Optimal solution with the minimum worst-case CVaR ($\nu = 163.7\,\mathrm{J}$)}
  \label{fig:Optimal_Min_Worst_CVaR_289-bar_Truss}
\end{subfigure}
\hfill
\begin{subfigure}[t]{0.45\textwidth}
\centering
\includegraphics[width=60mm]{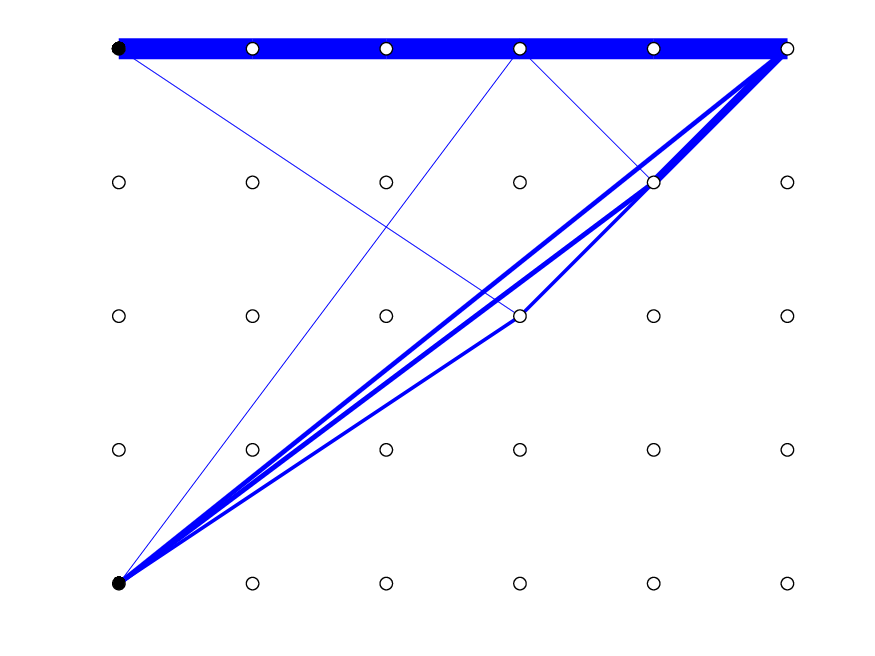}
\caption{Optimal solution with $\nu = 170\,\mathrm{J}$}
\label{fig:Optimal_185J_289-bar_Truss}
\end{subfigure}
\begin{subfigure}[t]{0.45\textwidth}
  \includegraphics[width=60mm]{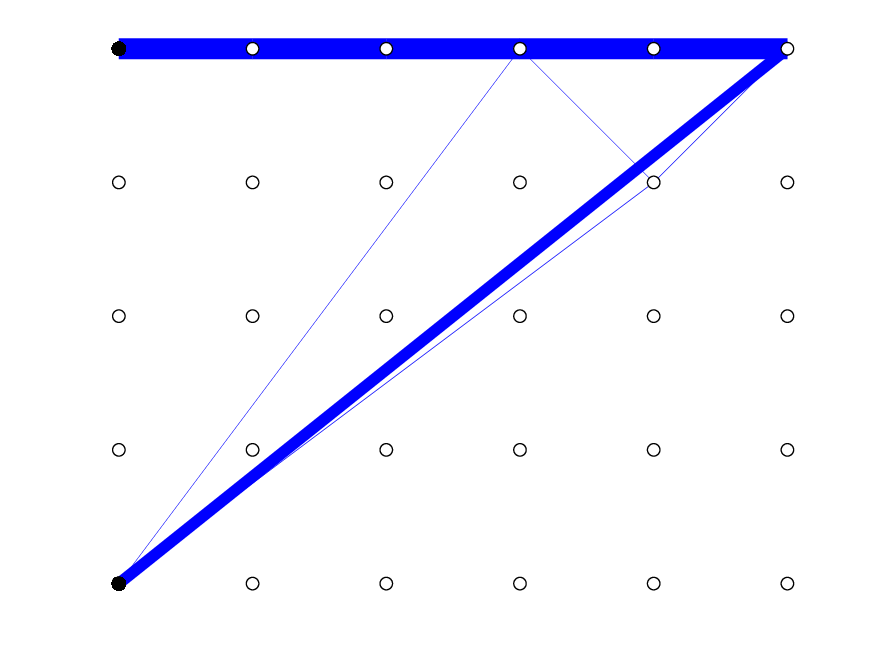}
  \caption{Optimal solution with $\nu = 185\,\mathrm{J}$}
  \label{fig:Optimal_175J_289-bar_Truss}
\end{subfigure}
\hfill
\begin{subfigure}[t]{0.45\textwidth}
\centering
\includegraphics[width=60mm]{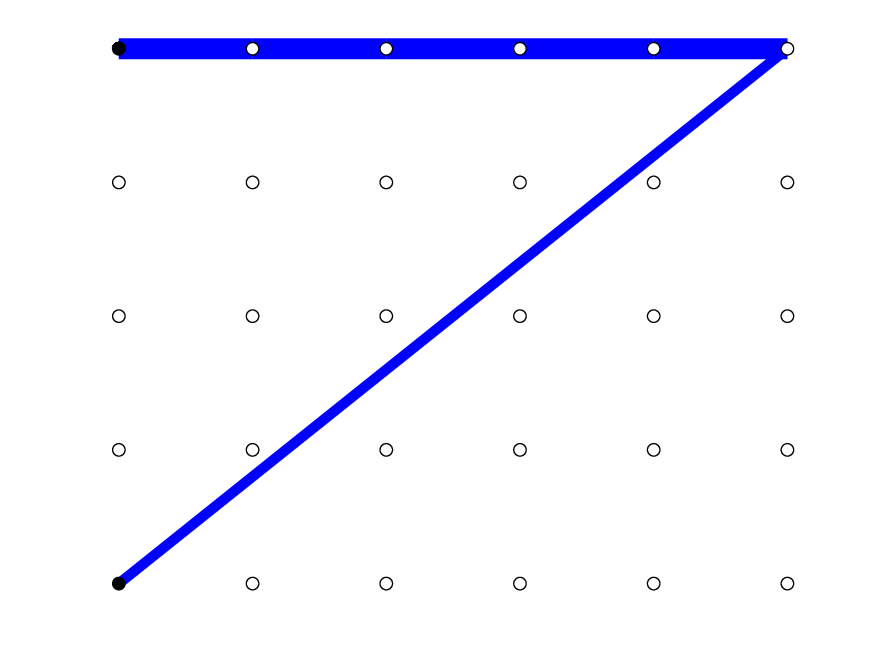}
\caption{Optimal solution with the minimum worst-case expected value ($\nu = 198.3\,\mathrm{J}$)}
\label{fig:Optimal_Min_Worst_expected_value_289-bar_Truss}
\end{subfigure}
\caption{Variation in the optimal solution with respect to the worst-case CVaR obtained in Section~6.2}
\label{fig:Optimal_solutions_289-bar_Truss}
\end{figure}

\subsection{Comparison between the uniform kernel and the triangular kernel} \label{6.3}
In this section, we conduct numerical experiments using the formulations with the uniform kernel presented in Section \ref{4} and the triangular kernel introduced in Section \ref{5}, aiming to compare the Pareto optimal solutions as well as the computational costs.
Consider the planar cantilever truss shown in Figure \ref{fig:289-bar_cantilever_truss}. The truss has $m=289$ members, and $d=50$ degrees of freedom for nodal displacements. In Figure \ref{fig:289-bar_cantilever_truss}, the filled  circles represent fixed nodes, while the open circles represent free nodes. The distance between the nearest nodes is $1\,\mathrm{m}$. The Young modulus of the members and the upper bound of the total member volume are $E=20\,\mathrm{GPa}$ and $\overline{V}=20000\,\mathrm{mm}^3$, respectively.
\begin{figure}
\centering
\includegraphics[width=60mm]{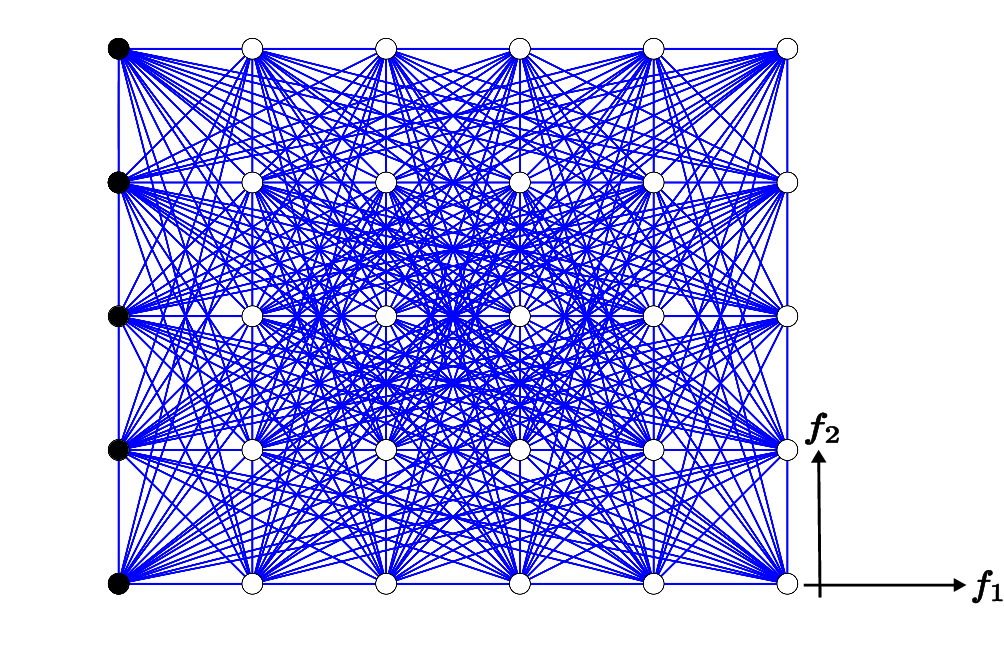}
\caption{289-bar cantilever truss in Section~6.3}
\label{fig:289-bar_cantilever_truss}
\end{figure}

We consider the case where the external forces act only on the bottom-right node of the truss depicted in Figure \ref{fig:289-bar_cantilever_truss}. The external forces follow a mixture distribution composed of two distributions. For the numerical experiments, $n=30$ samples were utilized. These samples were drawn from a mixture distribution consisting of $N(\bm{\mu}_4, \varSigma_4)$ and $N(\bm{\mu}_5, \varSigma_5)$, where
\begin{alignat*}{3}
  \bm{\mu}_4 &= 
  \begin{bmatrix}
    0 \\
    -100
  \end{bmatrix}\,\mathrm{kN}, &\quad
  \varSigma_4 &= 
  \begin{bmatrix}
    100 & 0 \\
    0 & 100
  \end{bmatrix}\,\mathrm{kN}^2, \\
  \bm{\mu}_5 &= 
    \begin{bmatrix}
        100 \\
        -100
    \end{bmatrix}\,\mathrm{kN}, &\quad
    \varSigma_5 &= 
    \begin{bmatrix}
        100 & 0 \\
        0 & 100
    \end{bmatrix}\,\mathrm{kN}^2.
\end{alignat*}
The sample mean vector $\hat{\bm{\mu}}_{4}$ and the sample variance-covariance matrix $\hat{\varSigma}_{4}$ obtained from the samples drawn from distribution $N(\bm{\mu}_4, \varSigma_4)$, as well as the sample mean vector $\hat{\bm{\mu}}_{5}$ and the sample variance-covariance matrix $\hat{\varSigma}_{5}$ obtained from the samples drawn from distribution $N(\bm{\mu}_5, \varSigma_5)$, are as follows:
\begin{alignat*}{3}
    \hat{\bm{\mu}}_{4} &= 
    \begin{bmatrix}
        -1.769 \\
        -97.632
    \end{bmatrix}\,\mathrm{kN}, &\quad
    &\hat{\varSigma}_{4} = 
    \begin{bmatrix}
        106.53 & -14.60 \\
        -14.60 & 111.54
    \end{bmatrix}\,\mathrm{kN}^2, \\
    \hat{\bm{\mu}}_{5} &= 
    \begin{bmatrix}
        96.723 \\
        -104.358
    \end{bmatrix}\,\mathrm{kN}, &\quad
    &\hat{\varSigma}_{5} = 
    \begin{bmatrix}
        67.56 & -3.12 \\
        -3.12 & 79.53
    \end{bmatrix}\,\mathrm{kN}^2.
\end{alignat*}

Figure \ref{fig:generated_load_data_289-bar_Cantilever_Truss} illustrates the samples used in the numerical experiments. Each ``$\circ$'' and ``$\times$'' correspond to the sample drawn from the normal distribution $N(\bm{\mu}_4, \varSigma_4)$ and the sample from $N(\bm{\mu}_5, \varSigma_5)$, respectively. The filled square and the filled triangle represent the mean vectors $\bm{\mu}_4$ and $\bm{\mu}_5$, respectively.

By solving the distributionally robust CVaR-constrained expected value minimization problem \eqref{Extension_Modified_Uniform_Compliance} using a uniform kernel, the Pareto front for the two objectives, i.e., the worst-case expected value and the worst-case CVaR of the compliance are obtained. The settings are as follows:
\begin{equation*}
 \gamma = 0.95,\quad h = 30\,\mathrm{J},\quad \bm{w}^0 = \frac{1}{30} \bm{1}.
\end{equation*}
\begin{figure}
\centering
\includegraphics[width=60mm]{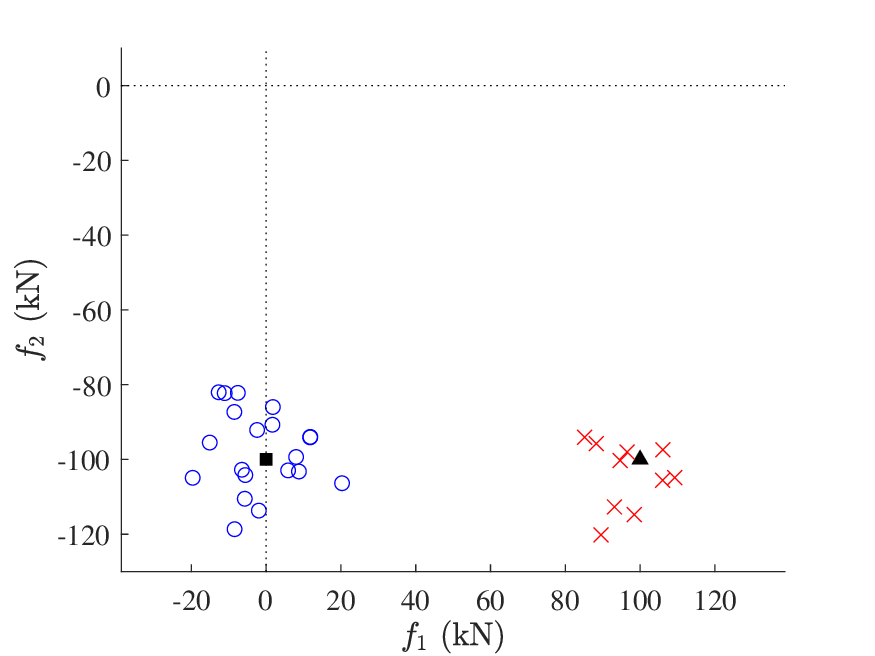}
\caption{30 samples used in Section~6.3}
\label{fig:generated_load_data_289-bar_Cantilever_Truss}
\end{figure}

Figure \ref{fig:Optimal_solutions_289-bar_Cantilever_Truss_Uni} and Figure \ref{fig:Optimal_solutions_289-bar_Cantilever_Truss_Tri} show how the optimal solution changes for $\tau=0.5$ with the uniform kernel and the triangular kernel, respectively. From the figures, it can be observed that the Pareto solutions of problem \eqref{Extension_Modified_Uniform_Compliance} using the uniform kernel and those of problem \eqref{Extension_Modified_Triangular_Compliance} using the triangular kernel show no significant differences in terms of the truss topology and the cross-sectional areas of the members in almost all cases. show no significant difference except the optimal solution with the minimum worst-case CVaR of the problem using the uniform kernel. In each figure, one can see that as the upper bound $\nu$ of the worst-case CVaR decreases, the volume of the members that are robust against horizontal variations in external force decreases, while the volume of the other members that are robust against vertical variations in external force increases.

\begin{figure}
\centering
\begin{subfigure}[t]{0.45\textwidth}
  \includegraphics[width=60mm]{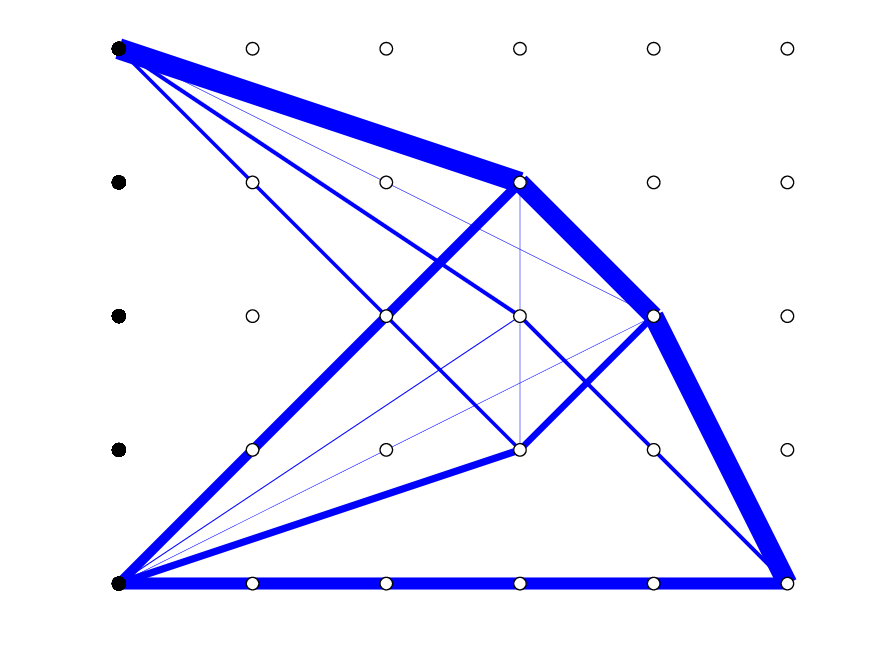}
  \caption{Optimal solution with the minimum worst-case CVaR ($\nu = 859.3\,\mathrm{J}$)}
  \label{fig:Optimal_Min_Worst_CVaR_289-bar_Cantilever_Truss_Uni}
\end{subfigure}
\hfill
\begin{subfigure}[t]{0.45\textwidth}
\centering
\includegraphics[width=60mm]{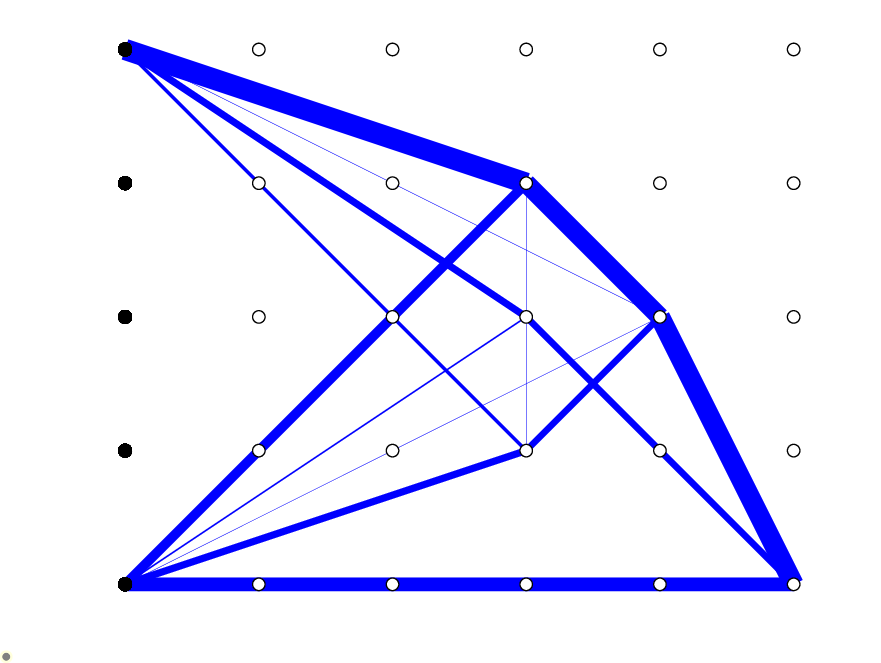}
\caption{Optimal solution with $\nu = 865\,\mathrm{J}$}
\label{fig:Optimal_850J_Cantilever_Truss_Uni}
\end{subfigure}
\begin{subfigure}[t]{0.45\textwidth}
  \includegraphics[width=60mm]{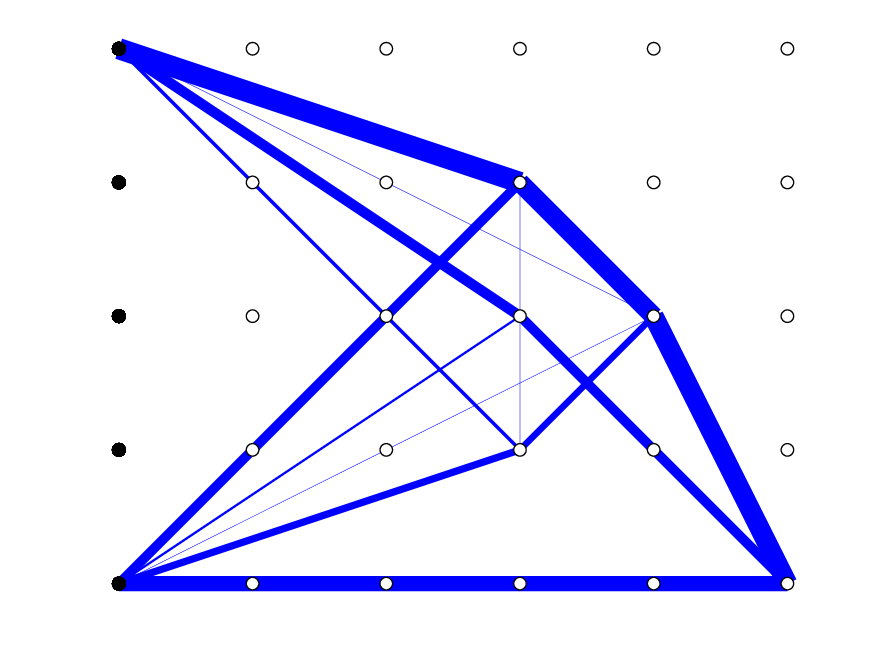}
  \caption{Optimal solution with $\nu = 870\,\mathrm{J}$}
  \label{fig:Optimal_852J_Cantilever_Truss_Uni}
\end{subfigure}
\hfill
\begin{subfigure}[t]{0.45\textwidth}
\centering
\includegraphics[width=60mm]{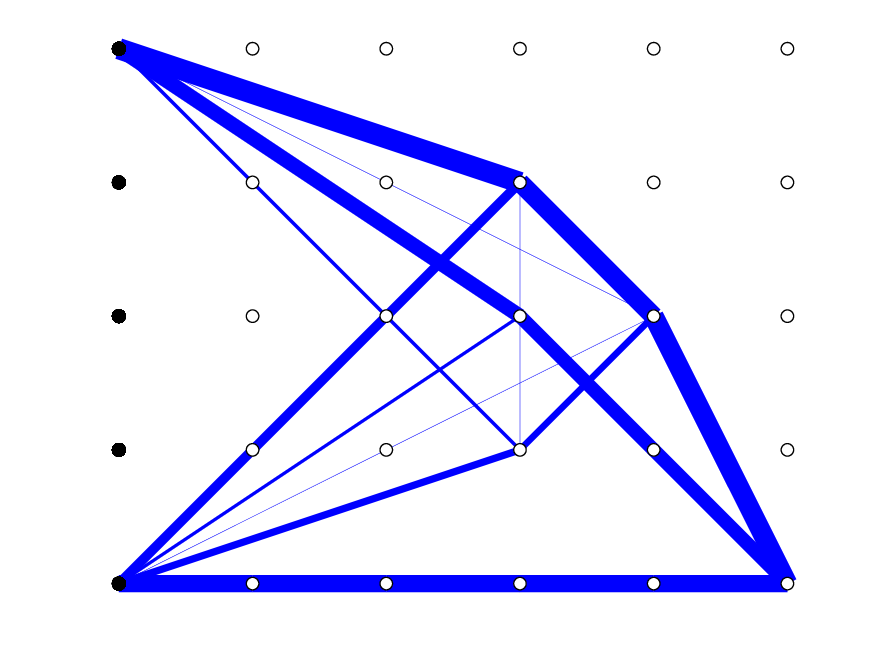}
\caption{Optimal solution with the minimum worst-case expected value ($\nu = 875.4\,\mathrm{J}$)}
\label{fig:Optimal_Min_Worst_expected_value_289-bar_Cantilever_Truss_Uni}
\end{subfigure}
\caption{Variation in the optimal solution with respect to the worst-case CVaR with uniform kernel obtained in Section~6.3}
\label{fig:Optimal_solutions_289-bar_Cantilever_Truss_Uni}
\end{figure}

\begin{figure}
\centering
\begin{subfigure}[t]{0.45\textwidth}
  \includegraphics[width=60mm]{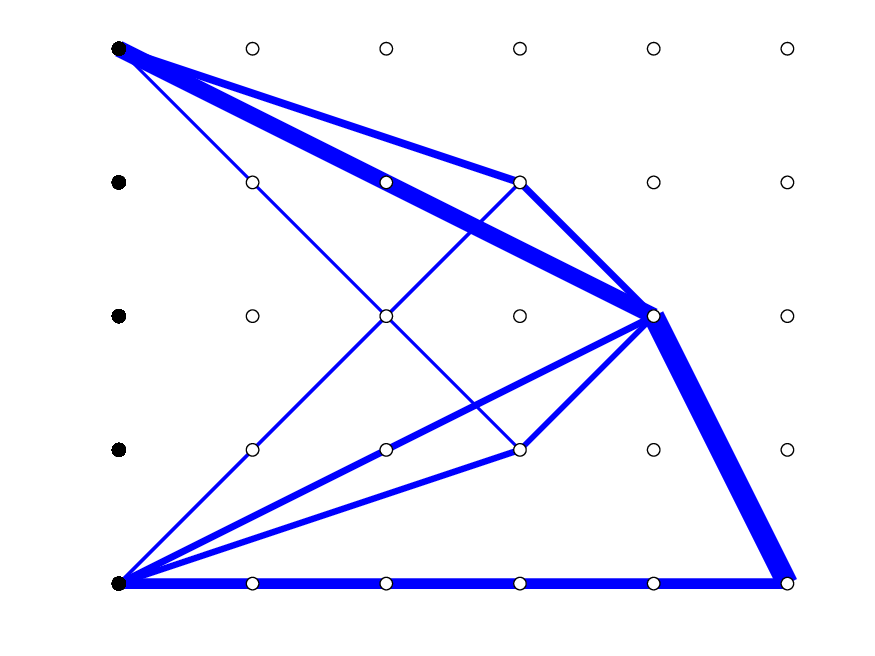}
  \caption{Optimal solution with the minimum worst-case CVaR ($\nu = 846.5\,\mathrm{J}$)}
  \label{fig:Optimal_Min_Worst_CVaR_289-bar_Cantilever_Truss_Tri}
\end{subfigure}
\hfill
\begin{subfigure}[t]{0.45\textwidth}
\centering
\includegraphics[width=60mm]{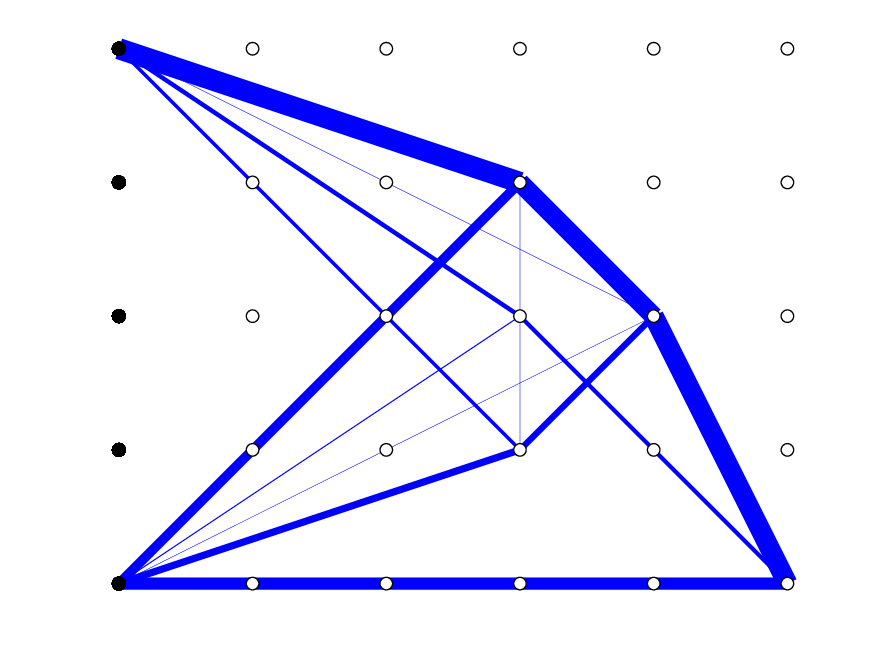}
\caption{Optimal solution with $\nu = 854\,\mathrm{J}$}
\label{fig:Optimal_185J_289-bar_Cantilever_Truss_Tri}
\end{subfigure}
\begin{subfigure}[t]{0.45\textwidth}
  \includegraphics[width=60mm]{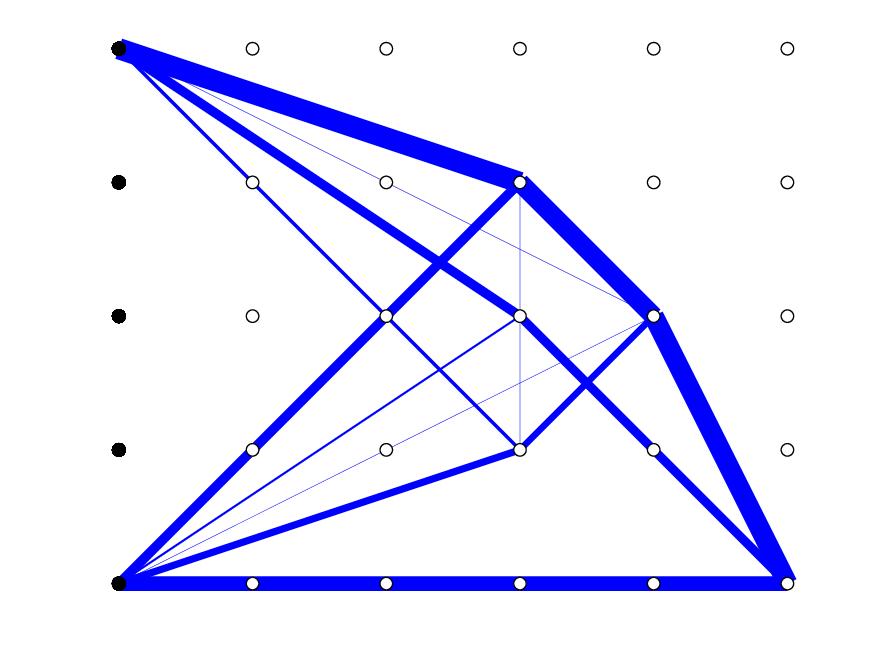}
  \caption{Optimal solution with $\nu = 862\,\mathrm{J}$}
  \label{fig:Optimal_175J_289-bar_Cantilever_Truss_Tri}
\end{subfigure}
\hfill
\begin{subfigure}[t]{0.45\textwidth}
\centering
\includegraphics[width=60mm]{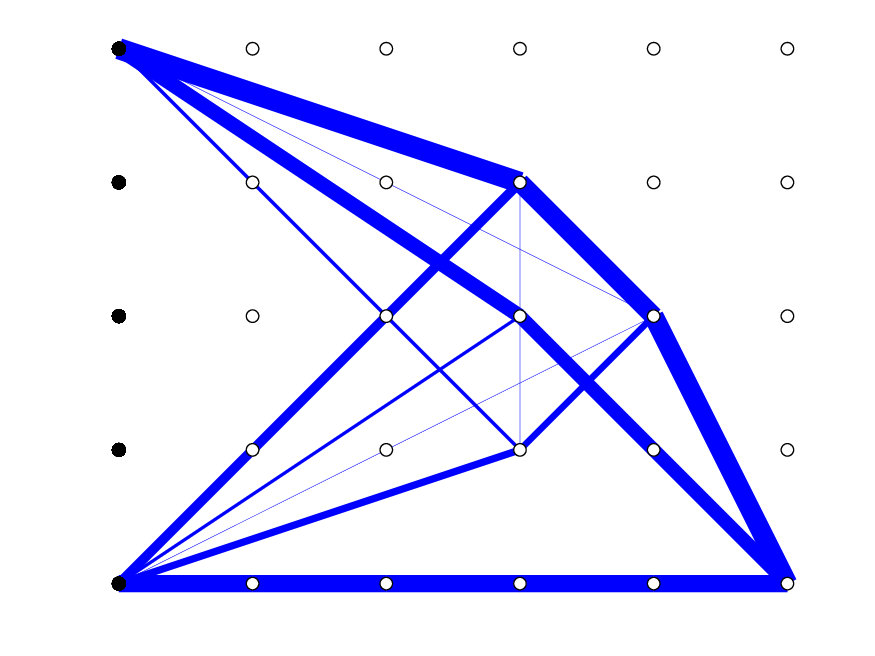}
\caption{Optimal solution with the minimum worst-case expected value ($\nu = 869.8\,\mathrm{J}$)}
\label{fig:Optimal_Min_Worst_expected_value_289-bar_Cantilever_Truss_Tri}
\end{subfigure}
\caption{Variation in the optimal solution with respect to the worst-case CVaR with triangular kernel obtained in Section~6.3}
\label{fig:Optimal_solutions_289-bar_Cantilever_Truss_Tri}
\end{figure}

Figure \ref{fig:Optimal_value_with_respect_to_1-gamma} shows the result of the comparison of the optimal values with respect to the confidence level between the uniform kernel \eqref{Extension_Modified_Uniform_Compliance} and the triangular kernel. In Figure \ref{fig:Optimal_value_with_respect_to_1-gamma}, the diamond and the plus signs represent the the optimal values with respect to the confidence level $1-\gamma$ in problem \eqref{Extension_Modified_Uniform_Compliance} and problem \eqref{Extension_Modified_Triangular_Compliance}, respectively. In both problems, as the confidence level increases, the feasible set of the design variables becomes smaller, resulting in an increase in the worst-case expected value. Due to the difference in the shape of the kernels, the feasible set of problem \eqref{Extension_Modified_Uniform_Compliance} is always included in the feasible set of problem \eqref{Extension_Modified_Triangular_Compliance}. As a result, when the same worst-case CVaR value is imposed on the constraint, the optimal value becomes smaller when using the triangular kernel.
\begin{figure}
\centering
\includegraphics[width=60mm]{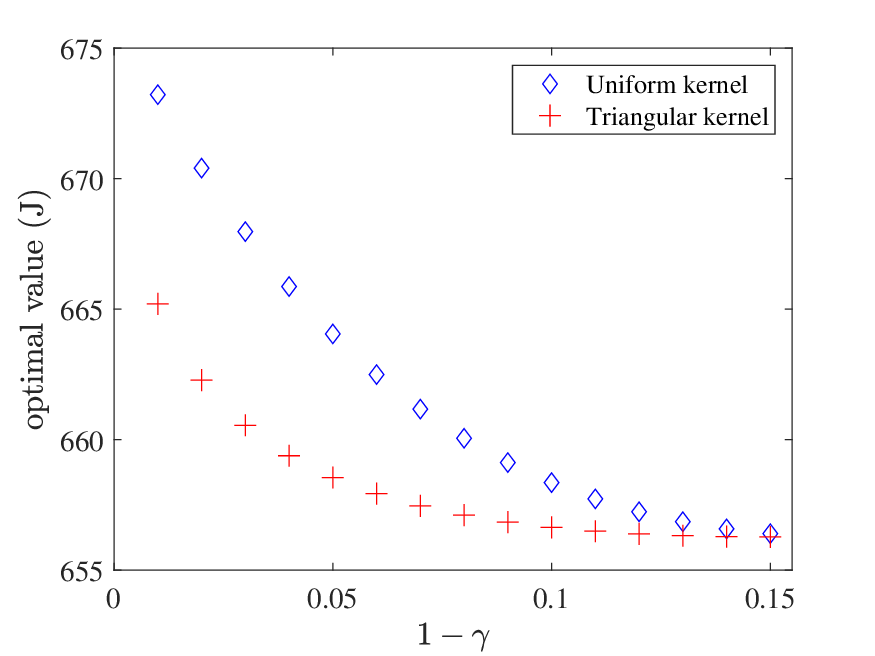}
\caption{Comparison of the optimal values with respect to the confidence level $1-\gamma$: the uniform kernel \eqref{Extension_Modified_Uniform_Compliance} and the triangular kernel \eqref{Extension_Modified_Triangular_Compliance} with $\nu = 859\,\mathrm{J}$}
\label{fig:Optimal_value_with_respect_to_1-gamma}
\end{figure}

Figure \ref{fig:Optimal_value_with_respect_to_tau} compares the optimal values with respect to the level of uncertainty $\tau$ between the uniform kernel formulation (problem \eqref{Extension_Modified_Uniform_Compliance}) and the triangular kernel formulation (problem \eqref{Extension_Modified_Triangular_Compliance}). In this figure, the diamond and the plus signs represent the optimal values of problem \eqref{Extension_Modified_Uniform_Compliance} and problem \eqref{Extension_Modified_Triangular_Compliance}, respectively. In both cases, increasing the uncertainty level enlarges the distributionally uncertainty set, which tightens the constraints and leads to a higher worst-case expected value. Figure~\ref{fig:Optimal_value_with_respect_to_tau} reveals that the difference in the optimal values resulting from the choice of kernel is relatively small although the optimal value becomes smaller when using the triangular kernel owing to the difference in kernel shapes.
\begin{figure}
\centering
\includegraphics[width=60mm]{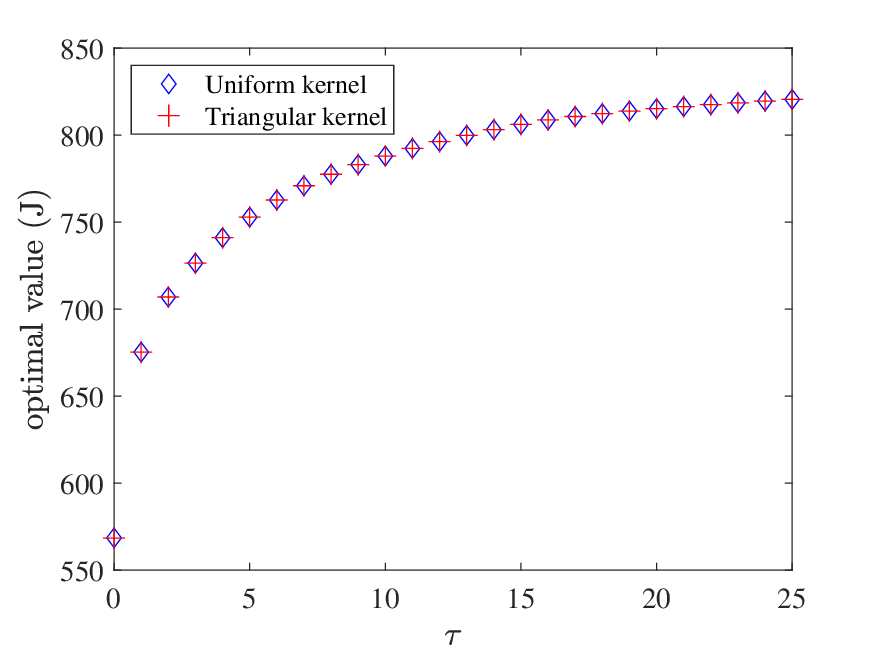}
\caption{Comparison of the optimal values with respect to the level of uncertainty $\tau$: the uniform kernel \eqref{Extension_Modified_Uniform_Compliance} and the triangular kernel \eqref{Extension_Modified_Triangular_Compliance} with $\nu = 859\,\mathrm{J}$}
\label{fig:Optimal_value_with_respect_to_tau}
\end{figure}

Figure \ref{fig:Computation_time_with_respect_to_the_number_of_samples} illustrates the result of the comparison of computation time between problem \eqref{Extension_Modified_Uniform_Compliance} with the uniform kernel and problem \eqref{Extension_Modified_Triangular_Compliance} with the triangular kernel with respect to the number of the samples. In Figure \ref{fig:Computation_time_with_respect_to_the_number_of_samples}, the diamond and the plus signs represent the computation time of problem \eqref{Extension_Modified_Uniform_Compliance} and problem \eqref{Extension_Modified_Triangular_Compliance}, respectively. The figure illustrates that as the sample size increases, the difference in computation time between problem \eqref{Extension_Modified_Triangular_Compliance}, and problem \eqref{Extension_Modified_Uniform_Compliance}, becomes more pronounced. This is attributed to the fact that the number of variables in problem \eqref{Extension_Modified_Triangular_Compliance} is greater than that in problem \eqref{Extension_Modified_Uniform_Compliance}.
\begin{figure}
\centering
\includegraphics[width=60mm]{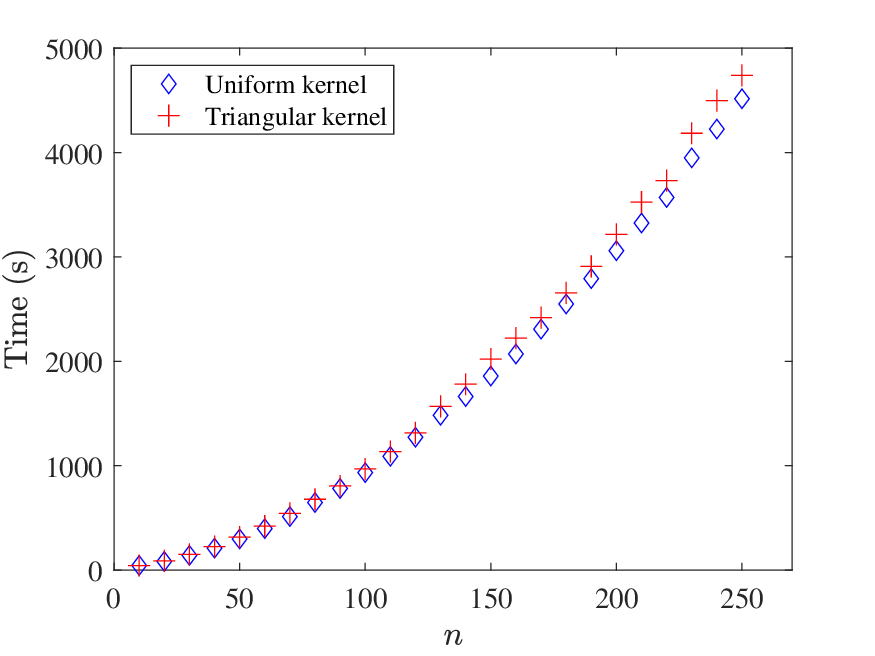}
\caption{Comparison of computation time with respect to sample size: the uniform kernel \eqref{Extension_Modified_Uniform_Compliance} and the triangular kernel \eqref{Extension_Modified_Triangular_Compliance}}
\label{fig:Computation_time_with_respect_to_the_number_of_samples}
\end{figure}

Figure \ref{fig:Computation_time_with_respect_to_the_number_of_members} illustrates the result of the comparison of computation time between problem \eqref{Extension_Modified_Uniform_Compliance} with the uniform kernel and problem  \eqref{Extension_Modified_Triangular_Compliance} with the triangular kernel with respect to the number of the members. The numerical experiment was conducted using the set of $30$ samples. In Figure \ref{fig:Computation_time_with_respect_to_the_number_of_members}, the diamond and the plus signs represent the computation time of problem \eqref{Extension_Modified_Uniform_Compliance} and problem \eqref{Extension_Modified_Triangular_Compliance}, respectively. The figure indicates that problems \eqref{Extension_Modified_Triangular_Compliance} and \eqref{Extension_Modified_Uniform_Compliance} remain computationally tractable, with problem instances involving around 2000 design variables being solved within one hour.
\begin{figure}
\centering
\includegraphics[width=60mm]{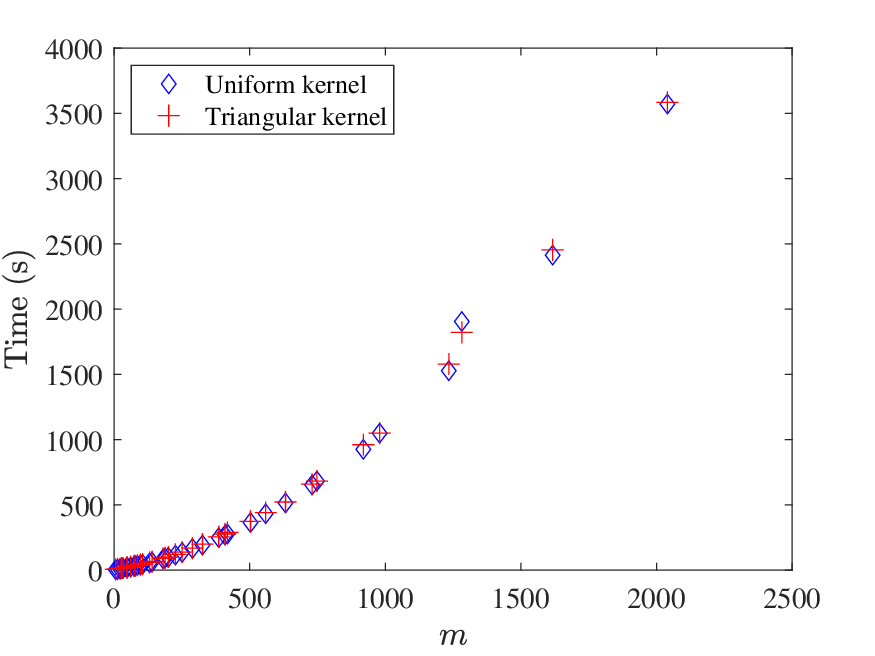}
\caption{Comparison of computation time with respect to the number of members: the uniform kernel \eqref{Extension_Modified_Uniform_Compliance} and the triangular kernel \eqref{Extension_Modified_Triangular_Compliance}}
\label{fig:Computation_time_with_respect_to_the_number_of_members}
\end{figure}

\section{Conclusion} \label{7}
In this paper, we formulated a bi-objective optimization problem for the worst-case expected value and the worst-case CVaR of the compliance of trusses, by using distributionally robust optimization and the risk measure CVaR. By employing the $\varepsilon$-constraint method, we derived a convex optimization problem minimizing the worst-case expected value under the worst-case CVaR constraint. Furthermore, we demonstrated that the derived problem can be reduced to a second-order cone programming (SOCP) problem when adopting either the uniform kernel or the triangular kernel as the kernel function for the kernel density estimation. As a result, the derived formulation ensures global optimality and allows for polynomial-time solutions using off-the-shelf SOCP solvers. These properties offer practical benefits in structural engineering. In particular, the SOCP formulation ensures efficient computation, making the method applicable even to large-scale structural systems.

The proposed formulation is expected to have two advantages over existing methods for reliability-based design optimization with confidence level. First, since it does not approximate the performance function, it is expected to achieve higher accuracy for problems involving highly nonlinear performance functions compared to existing methods. Second, while constraints in reliability-based design optimization with confidence levels correspond to constraints on VaR, the proposed formulation imposes a constraint on CVaR. This allows for the consideration of risk associated with the tail of the performance function, which cannot be accounted for by VaR, particularly in the design of structures that may experience significant performance degradation upon severe damage.

Through numerical experiments, we obtained the Pareto front for the bi-objective optimization of the worst-case expected value and the worst-case CVaR of the compliance. By analyzing variation in the Pareto solutions, we confirmed that the optimal topology of trusses changes according to the trade-off relation between the two objective functions. 

Future work includes investigating a principled, data-driven appproaches for calibrating the kernel bandwidth $h$ in kernel density estimation and the uncertainty level parameter $\tau$. Another important direction is the extension of the proposed methodology to constraints beyond compliance as well as to structural systems other than trusses. It is also important to incorporate other sources of uncertainty. For instance, when the uncertainty lies in material stiffness, the compliance remains linear with respect to the Young's modulus, as indicated by equation \eqref{decomp}, which suggests that a convex formulation may still be possible. However, deriving such a formulation explicitly remains an open challenge. In contrast, accounting for uncertainty in the design variables is considerably less straightforward, and incorporating uncertainty in the nodal positions is particularly difficult due to the nonlinearity it introduces.

\paragraph{Author Contributions}
All authors contributed to the study conception and design. T.~F. derived specific optimization formulations, developed and implemented the algorithms, conducted numerical experiments, analyzed the trade-off relation between objectives, and wrote manuscript. Y.~K. supervised the research, provided guidance on the research direction, helped to develop method, and provided manuscript refinement.
\paragraph{Funding}
The work of the second author is partially supported by 
JSPS KAKENHI JP24K07747. 
\paragraph{Availability of data and materials}
The datasets generated during the current study, as well as the source codes for the numerical experiments, are available in the GitHub repository: \url{https://github.com/Takumi-Fujiyama/distributionally-robust-truss-compliance}.
\section*{Declarations}
\paragraph{Conflicts of interest}
The authors declare no conflicts of interest to report regarding the present study. 
\paragraph{Consent for publication}
All authors consent to the publication of this manuscript.
\paragraph{Ethics approval and consent to participate}
Not Applicable.


\appendix
\section{Proofs}
\subsection{Proof of Proposition \ref{prop_convex}} \label{A.1}
First, we prepare the property of preserving convexity for composite functions through the following lemma.
\begin{lemma} \textup{\cite[Theorem~5.1]{Rockafellar70}} \label{Convex Composite} \\
Let $F: \mathbb{R}^n \to \mathbb{R} \cup \{ + \infty\}$ be a convex 
  function, and let $\psi: \mathbb{R} \to \mathbb{R} \cup \{ +\infty\}$ 
  be a non-decreasing convex function. Then $\psi(F(\cdot))$ is convex 
  on $\mathbb{R}^n$, where we use the convention $\psi(+\infty) = + \infty$.
\end{lemma}

Next, we provide the convexity of the compliance of trusses as follows. 
\begin{lemma} \textup{\cite[Theorem~A-D]{ABBZ92}} \label{Convex Compliance} \\
  Let $\mathcal{X} \subset \mathbb{R}^{m}$ 
  be defined by \eqref{X}, and $\bm{\xi} \in \mathbb{R}^d$ be a constant vector. 
  Then, $\pi^\textrm{c}(\,\cdot\,; \bm{\xi}):\mathcal{X} 
  \to \mathbb{R}\cup\{+\infty\}$ defined by 
  \eqref{compliance} is a convex function.
\end{lemma}

Furthermore, we provide the convexity of the conjugate function as 
follows. 
\begin{lemma} \textup{\cite[Section 3.3.1]{BV04}} \label{Convex Conjugate} \\
  Let $f: \mathbb{R}^n \to \mathbb{R}$ be a convex function. Then, its conjugate function $f^*: \mathbb{R}^n \to \mathbb{R} \cup \{ + \infty\}$ defined by
\begin{align*}
  f^*(\bm{s}) = \sup \{ \bm{s}^\top \bm{t} - f(\bm{t}) \}
\end{align*}
is convex.
\end{lemma}

\paragraph{Proof of Proposition \ref{prop_convex}}
  It follows from Lemma \ref{Convex Compliance} that 
  $\pi^\textrm{c}(\bm{x}; \hat{\bm{\xi}}_i) - \alpha$ for each 
  $i=1, \dots, n$ is convex with respect to $(\bm{x}, \alpha)$. 
  Therefore, application of Lemma \ref{Convex Composite} 
  with $n:=m+1$ and Proposition \ref{Upsilon_Prop} shows that $\Upsilon_k ( \pi^\textrm{c}(\bm{x}; \hat{\bm{\xi}}_i) - \alpha )$ is a convex function with respect to $(\bm{x}, \alpha)$. 
  The convexity $\phi^*$ follows from Lemma \ref{Convex Conjugate}, 
  which completes the proof. 
  \hfill $\qed$

\subsection{Proof of Proposition \ref{nes_suf_condition}} \label{A.2}
\begin{proof}
  It follows from \eqref{modiefied_conjugate} that $z \geq \phi^*(y)$ 
  can be rewritten as
\begin{align}
 z \geq \frac{{({[y + 2]}^+)}^2}{4} - 1. \label{Prop.3.3}
\end{align}
The term ${[y + 2]}^+$ on the right-hand side of this inequality can be expressed using a new variable $a \in \mathbb{R}$ as
\begin{align*}
  {\left[ y+2 \right]}^+ = \min \{ a \mid a \geq y+2, a \geq 0 \}.
\end{align*}
Therefore, \eqref{Prop.3.3} holds if and only if there exists an $a \in \mathbb{R}$ satisfying
\begin{align}
  z \geq \frac{{a}^2}{4}-1,\quad a \geq y+2,\quad a \geq 0. \label{reform_modified2}
\end{align}
Moreover, the first inequality in \eqref{reform_modified2} is equivalent 
  to the following second-order cone constraint:
\begin{align*}
  z + 1 \geq
  {\left\| 
  \begin{bmatrix}
   z - 1  \\
   a  \\
  \end{bmatrix}
  \right\|},
\end{align*}
  which completes the proof. 
\end{proof}

\subsection{Proof of Proposition \ref{Uniform_decomp}} \label{A.3}
\begin{proof}
  Substitution of \eqref{Uniform_kernel} into \eqref{Upsilon_decomp} yields 
  \begin{align*}
  \Upsilon_k(c) = c\int_{-\infty}^{\frac{c}{h}} \frac{1}{2} \1_{\{|y| \leq 1\}} \text{d}y - h \int_{-\infty}^{\frac{c}{h}} \frac{1}{2} y \1_{\{|y| \leq 1\}} \text{d}y.
\end{align*}
By performing the integration, we obtain
\begin{align} \label{Upsilon_Uniform}
\Upsilon_k (c) =
\begin{cases}
  0 & \text{if}\ c < -h, \\
  \displaystyle \frac{{(c+h)}^2}{4h} & \text{if}\ -h \leq c < h, \\
  c & \text{if}\ c \geq h.
\end{cases}
\end{align}

Next, following \textup{\cite[Section~4.3.2]{Kannosan}}, 
consider the additive decomposition of $c$ given by  
\begin{align*}
  c = c_\rma + c_\rmq + c_\rms,
\end{align*}
where $c_\rma$ and $c_\rmq$ correspond to the linear and quadratic parts of $\Upsilon_k(c)$, respectively, and $c_\rms$ corresponds to the constant part. 
Then we see that \eqref{Upsilon_Uniform} can be rewritten as the optimal value of the following optimization problem:
\begin{alignat*}{3}
&\textrm{Min.} &\quad &c_\rma + \frac{{c_\rmq}^2}{4h} \\
&\st &\quad &c_\rma + c_\rmq \geq c+h, \\
& &\quad &0 \leq c_\rmq \leq 2h,\quad c_\rma \geq 0. 
\end{alignat*}
Moreover, minimizing $\displaystyle \frac{{c_\rmq}^2}{4h}$ is equivalent to minimizing the variable $s$ under the constraint
\begin{align*}
  s \geq \frac{{c_\rmq}^2}{4h},
\end{align*}
and this constraint can be rewritten as the second-order cone constraint
\begin{align*}
  s + h \geq
    {\left\| 
  \begin{bmatrix}
   s - h  \\
   c_\rmq \\
  \end{bmatrix}
  \right\|}.
\end{align*}
Therefore, $\Upsilon_k(c)$ in \eqref{Upsilon_Uniform} can finally be 
  expressed as the optimal value of problem 
  \eqref{Uniform Upsilon}, which is a 
  the second-order cone programming problem. 
\end{proof}

\subsection{Proof of Proposition \ref{Triangular_decomp}} \label{A.4}
\begin{proof}
  Substitution of \eqref{Triangular_kernel} into \eqref{Upsilon_decomp} 
  yields 
\begin{align*}
  \Upsilon_k(c) = c\int_{-\infty}^{\frac{c}{h}} (1 - |y|) \1_{\{|y| \leq 1\}} \text{d}y - h \int_{-\infty}^{\frac{c}{h}} y(1 - |y|) \1_{\{|y| \leq 1\}} \text{d}y.
\end{align*}
By performing the integration, we obtain
\begin{align} \label{Upsilon_Triangular}
\Upsilon_k (c) =
\begin{cases}
  0 & \text{if}\ c < -h, \\
  \displaystyle \frac{{(c+h)}^3}{6h^2} & \text{if}\ -h \leq c < 0, \\
  \displaystyle \frac{{(h-c)}^3}{6h^2} + c & \text{if}\ 0 \leq c < h, \\
  c & \text{if}\ c \geq h.
\end{cases}
\end{align}

Next, we consider the additive decomposition of variable $c$ given by 
\begin{align*}
  c = c_\rma + c_{\mathrm{c1}} + c_{\mathrm{c2}} + c_\rms,
\end{align*}
where $c_\rma$ corresponds to the linear part of $\Upsilon_k(c)$, $c_{\mathrm{c1}}$ and $c_{\mathrm{c2}}$ correspond to that of the cubic part, and $c_\rms$ correspond to that of the constant part. 
Then we see that $\Upsilon_k(c)$ in \eqref{Upsilon_Triangular} can be written as the optimal value of the following optimization problem:
\begin{equation*}
\begin{array}{ll}
    \textrm{Min.} & c_\rma + \displaystyle \frac{{c_{\mathrm{c1}}}^3}{6h^2} + \frac{{c_{\mathrm{c2}}}^3}{6h^2} - c_{\mathrm{c2}} + \frac{5h}{6} \\
    \st & c_\rma + c_{\mathrm{c1}} - c_{\mathrm{c2}} \geq c, \\
    & 0 \leq c_{\mathrm{c1}} \leq h,\quad 0 \leq c_{\mathrm{c2}} \leq h,\quad c_\rma \geq 0.
\end{array}
\end{equation*}
In this formulation, minimizing ${c_{\mathrm{c1}}}^3$ under the constraint $c_{\mathrm{c1}} \geq 0$ is equivalent to minimizing variable $s_1$ under the constraint
\begin{align*}
  s_1 \geq {c_{\mathrm{c1}}}^3,\quad c_{\mathrm{c1}} \geq 0,
\end{align*}
which is equivalent to
\begin{align}
  s_1 c_{\mathrm{c1}} \geq {c_{\mathrm{c1}}}^4,\quad c_{\mathrm{c1}} \geq 0. \label{ineq1}
\end{align}
Furthermore, constraint \eqref{ineq1} can be transformed with a new variable $r_1 \in \mathbb{R}$ as
\begin{align}
  s_1 c_{\mathrm{c1}} \geq {r_1}^2,\quad r_1 \geq c_{\mathrm{c1}}^2,\quad c_{\mathrm{c1}} \geq 0, \label{ineq2}
\end{align}
which can be rewritten as the second-order constraints 
\begin{equation*}
  s_1 + c_{\mathrm{c1}} \geq
    {\left\| 
  \begin{bmatrix}
    s_1 - c_{\mathrm{c1}}  \\
    2r_1  \\
  \end{bmatrix}
  \right\|},\quad 
  r_1 + \frac{1}{4} \geq
    {\left\| 
  \begin{bmatrix}
   r_1 - \frac{1}{4}  \\
   c_{\mathrm{c1}}  \\
  \end{bmatrix}
  \right\|}, \quad
  c_{\mathrm{c1}} \geq 0.
\end{equation*}
Similarly, minimizing ${c_{\mathrm{c2}}}^3$ under the constraint  $c_{\mathrm{c2}} \geq 0$ is equivalent to minimizing $s_2$ under the constraint
\begin{align}
  s_2 \geq {c_{\mathrm{c2}}}^3,\quad c_{\mathrm{c2}} \geq 0, \label{ineq10}
\end{align}
which is equivalent to
\begin{align}
  s_2 c_{\mathrm{c2}} \geq {c_{\mathrm{c2}}}^4,\quad c_{\mathrm{c2}} \geq 0. \label{ineq3}
\end{align}
Furthermore, the constraint \eqref{ineq3} can be transformed with a new variable $r_2 \in \mathbb{R}$ as
\begin{align*}
  s_2 c_{\mathrm{c2}} \geq {r_2}^2,\quad r_2 \geq c_{\mathrm{c2}}^2,\quad c_{\mathrm{c2}} \geq 0,
\end{align*}
which can be rewritten as the second-order cone constraints 
\begin{equation*}
  s_2 + c_{\mathrm{c2}} \geq
    {\left\| 
  \begin{bmatrix}
   s_2 - c_{\mathrm{c2}} \\
   2r_2  \\
  \end{bmatrix}
  \right\|},\quad 
  r_2 + \frac{1}{4} \geq
    {\left\| 
  \begin{bmatrix}
   r_2 - \frac{1}{4} \\
   c_{\mathrm{c2}} \\
  \end{bmatrix}
  \right\|}, \quad
   c_{\mathrm{c2}} \geq 0.
\end{equation*}
Moreover, minimizing $-c_{\mathrm{c2}} + \displaystyle \frac{5h}{6}$ under the constraint $c_{\mathrm{c2}} \geq 0$ is equivalent to minimizing $s_3$ under the constraint
\begin{align*}
  s_3 \geq -c_{\mathrm{c2}} + \displaystyle \frac{5h}{6},\quad  c_{\mathrm{c2}} \geq 0,
\end{align*}
which can be transformed as
\begin{align*}
  s_3 + c_{\mathrm{c2}} \geq \frac{5h}{6},\quad c_{\mathrm{c2}} \geq 0.
\end{align*}
Therefore, $\Upsilon_k(c)$ in \eqref{Upsilon_Triangular} can be finally  
  expressed as the optimal value of problem \eqref{Triangular Upsilon}, 
  which is a second-order cone programming problem. 
\end{proof}

\end{document}